\documentclass[reqno,11pt]{amsart}

\usepackage{amsmath,amssymb,amsfonts,amsthm, amscd,indentfirst}
\usepackage{amsmath,latexsym,amssymb,amsmath,
	amscd,amsthm,amsxtra,mathrsfs}
 
\usepackage[top=3cm, bottom=3.5cm, right=2.5cm, left=2.2cm]{geometry}

\usepackage[hyphens]{url} 
\usepackage{hyperref}
\usepackage{bookmark}
\usepackage{amsmath,thmtools,mathtools}
\usepackage{etoolbox}

\allowdisplaybreaks 
\mathtoolsset{showonlyrefs=true}
\DeclareRobustCommand{\textsection}{\ifmmode\mathsection\else\S\fi}

\usepackage[
  backend=biber,
  style=numeric,
  sorting=nyt,
  maxbibnames=99,
  giveninits=true,
]{biblatex}
\renewbibmacro{in:}{}

\DeclareFieldFormat[article,inproceedings,incollection,unpublished]{title}{#1}
\DeclareFieldFormat[article,inproceedings,incollection,unpublished]{citetitle}{#1}

\DeclareSourcemap{%
  \maps[datatype=bibtex]{%
    \map{\step[fieldsource=shortjournal, fieldtarget=journaltitle]}%
  }%
}

\makeatletter
\newcommand*\bbx@lasthash{}
\newtoggle{bbx@dashed}
\AtBeginBibliography{\global\let\bbx@lasthash\empty}
\AtEveryBibitem{%
  \global\togglefalse{bbx@dashed}%
  \iffieldundef{fullhash}
    {\global\let\bbx@lasthash\empty}
    {\iffieldequalstr{fullhash}{\bbx@lasthash}{\global\toggletrue{bbx@dashed}}{}%
     \xdef\bbx@lasthash{\thefield{fullhash}}}%
}
\DeclareNameFormat{labelname}{%
  \ifboolexpr{togl {bbx@dashed}}
    {\ifnum\value{listcount}=1\relax\bibnamedash\fi}
    {\nameparts{#1}%
     \usebibmacro{name:family-given}
       {\namepartfamily}
       {\namepartgiven}
       {\namepartprefix}
       {\namepartsuffix}%
     \usebibmacro{name:andothers}}%
}
\makeatother

\newcounter{mparcnt}

\usepackage{fancyhdr}
\usepackage{esint}
\usepackage{enumerate}
\usepackage{xcolor}

\usepackage{pictexwd,dcpic}
\usepackage{graphicx}
\usepackage{caption}

\usepackage{graphicx}
\usepackage{caption}
\usepackage{slashed}

\usepackage{epsfig,here}
\usepackage{subfigure,here}

\newtheorem{theorem}{Theorem}[section]
\newtheorem{lemma}[theorem]{Lemma}
\newtheorem{proposition}[theorem]{Proposition}
\newtheorem{definition}[theorem]{Definition}
\newtheorem{corollary}[theorem]{Corollary}
\newtheorem{remark}[theorem]{Remark}

\newcommand{\abs}[1]{\lvert#1\rvert}
\newcommand{\Abs}[1]{\left\lvert#1\right\rvert}

\newcommand{\norm}[1]{\lVert#1\rVert}

\newcommand{\rd}{{\rm d}}
\newcommand{\rdV}{{\rm dV}}

\newcommand{\rVol}{{\rm Vol}}

\newcommand{\rid}{{\rm id}}
\newcommand{\rtr}{{\rm tr}}
\newcommand{\rdiv}{{\rm div}}

\newcommand{\D}{{\slashed{D}}}
\newcommand{\rRe}{{\rm Re}}

\newcommand{\curl}{{\rm curl}\,}
\newcommand{\ip}{\lrcorner\,}
\newcommand{\w}{\wedge}

\def\<{\langle}
\def\>{\rangle}
\def\S{\mathbb{S}}
\def\R{\mathbb{R}}

\newcommand{\ra}{\rightarrow}

\def\SS{{\mathbb S}}

\newcommand{\eq}[1]{\begin{equation}\allowdisplaybreaks\begin{alignedat}{2} #1 \end{alignedat}\end{equation}}

\numberwithin{equation} {section}

\begin{document}

	
\title[The sharp curl--Sobolev inequality]
{The sharp curl--Sobolev inequality 
}
\date{\today}

\author{Guofang Wang}
\address{ Albert-Ludwigs-Universit\"at Freiburg,
Mathematisches Institut,
Ernst-Zermelo-Str. 1,
D-79104 Freiburg, Germany}
\email{guofang.wang@math.uni-freiburg.de}

\author{Mingwei Zhang}
\address{ Wuhan University, School of Mathematics and Statistics, 430072 Wuhan, China and 
Albert-Ludwigs-Universit\"at Freiburg,
Mathematisches Institut,
Ernst-Zermelo-Str. 1,
D-79104 Freiburg, Germany}
\email{zhangmwmath@whu.edu.cn}

\begin{abstract}
We solve a longstanding problem, going back at least to Rivi\`ere~\cite{Riviere98CAG} and open even in the physically most relevant case $n=3$, by proving a sharp curl--Sobolev inequality for differential forms on the round sphere $\S^n$ when $n\equiv 3\pmod 4$. More precisely, for every $\frac{n-1}{2}$-form $\alpha$ on $\S^n$ with $\int_{\S^n}\<\curl\alpha,\alpha\>\,\rdV>0$, the conformally invariant quotient satisfies
\[
    \frac{\Bigl(\int_{\S^n}\abs{\curl\alpha}^{\frac{2n}{n+1}}\,\rdV\Bigr)^{\frac{n+1}{n}}}{\int_{\S^n}\<\curl\alpha,\alpha\>\,\rdV}
    \ge \frac{n+1}{2}\,\omega_n^{\frac1n}.
\]
We also classify all extremals: equality holds if and only if $\alpha$ is a positive Killing $\frac{n-1}{2}$-form, modulo orientation-preserving conformal transformations and modulo $\ker(\rd)$. By conformal invariance, the same sharp inequality (with the same constant and the corresponding class of extremals up to pullback) holds on $\R^n$.

We then give geometric and variational applications that settle several open problems and conjectures in geometry and mathematical physics. First, for the conformal invariant
\[
    \mu([g])\coloneqq \inf_{\tilde g\in[g]}\lambda_1^+(\tilde g)\,\rVol(\tilde g)^{1/n},
\]
where $\lambda_1^+(\tilde g)$ denotes the first positive curl eigenvalue, we show that on $\S^n$ the round metric is the unique optimizer for $\mu([g])$ (a Hersch/El Soufi--Ilias type theorem for curl). Second, we prove that the unique minimizers of the $3$-energy $\int_{\S^3}\abs{\rd u}^3$ in the homotopy class of the Hopf map $\pi:\S^3\to\S^2$ are exactly $\pi\circ\Phi$ with $\Phi\in{\rm Conf}^+(\S^3)$, confirming a conjecture of Rivi\`ere~\cite{Riviere98CAG}. Third, for the Faddeev--Skyrme energy $\mathcal{FS}_\rho$ on $\S^3$, we establish global minimality of the Hopf map in the full predicted range: for every coupling constant $\rho\le \sqrt{2}$, the unique global minimizers in its homotopy class are precisely $\pi\circ R$ with $R\in\mathrm{SO}(4)$, as expected since Ward~\cite{Ward99}. Fourth, in the presence of Dirac zero modes on $\S^3$, we prove the sharp lower bound $\norm{\curl A}_{3/2} \ge 3\omega_3^{2/3}$ for the magnetic field and characterize equality in terms of Killing spinors; in particular, this yields a sharp criterion for the existence of zero modes and answers a question of Frank--Loss~\cite{FL1} for $n=3$.


\end{abstract}

\maketitle
\tableofcontents

\section{Introduction}

\bigskip

We resolve the following longstanding conjecture.

\begin{theorem}\label{main_thm}
Let $n\equiv 3\pmod 4$, and let $\alpha$ be an $\frac{n-1}{2}$-form on $\S^n$. Set
\eq{
    J(\alpha) \coloneqq \frac{ \Big(\int_{\S^n} \abs{\curl\alpha}^{\frac{2n}{n+1}} \,\rdV \Big)^{\frac{n+1}{n}} }{ \int_{\S^n} \< \curl\alpha, \alpha\> \,\rdV},
}
then
\eq{
  \inf\left\{J(\alpha)\,\middle|\, \int_{\S^n} \< \curl\alpha,\alpha\>>0\right\}
  = \frac{n+1}{2}\omega_n^{\frac{1}{n}},
}
where $\omega_n$ is the area of the standard round sphere $\S^n$.
Moreover, the minimizers are precisely positive Killing $\frac{n-1}{2}$-forms modulo orientation-preserving conformal transformations and modulo $\ker(\rd)$. Here $\curl\coloneqq\ast \rd$ is the curl operator.
\end{theorem}

We remark that in the case $n\equiv1\pmod4$, one can instead define the curl operator by $\curl\coloneqq i\ast\rd$, where $i$ is the imaginary unit, in order to preserve self-adjointness. 
Using the same argument, one can analogously obtain the same result as Theorem~\ref{main_thm} for complex-valued differential forms.
In this paper, we focus on the case $n\equiv3\pmod4$.

Since $J(\alpha)$ is conformally invariant (see Subsection~\ref{sec2.2}), Theorem~\ref{main_thm} is equivalent (via stereographic projection) to the corresponding sharp inequality on $\R^n$:
\eq{
    \frac{ \Big(\int_{\R^n} \abs{\curl\alpha}^{\frac{2n}{n+1}} \,\rdV \Big)^{\frac{n+1}{n}} }{ \int_{\R^n} \< \curl\alpha, \alpha\> \,\rdV} \geq \frac{n+1}{2}\omega_n^{\frac{1}{n}}.
}

The case $n=3$ is particularly important for applications (notably in physics). For convenience, we restate the result for vector fields, using the classical notation $\curl=\nabla\times$.
\begin{corollary}\label{thm_form_3D}
Let $n=3$. For any vector field $v$ on $\S^3$, we have
\eq{
   \frac{ \Big(\int_{\S^3} \abs{\nabla \times v }^{\frac{3}{2}} \,\rdV \Big)^{\frac{4}{3}} }{ \int_{\S^3} \< \nabla\times v, v\> \,\rdV} \ge 2\omega_3^{\frac 13}, \quad\hbox{whenever}\quad \int_{\S^3} \< \nabla\times v, v\> \,\rdV > 0,
}
with equality if and only if $v$ is a positive Killing vector field, modulo orientation-preserving conformal transformations and modulo $\ker(\nabla\times)$.
Equivalently, the following inequality holds on $\R^3$:
\eq{\label{eq:curl}
   \frac{ \Big(\int _{\R^3}\abs{\nabla \times v }^{\frac{3}{2}} \,\rdV \Big)^{\frac{4}{3}} }{ \int _{\R^3}\< \nabla\times v, v\> \,\rdV} \ge 2\omega_3^{\frac 13},
}
with equality if and only if
\eq{
    v = \frac{3}{(1+\abs{x}^2)^2}  \begin{pmatrix}
        1+x_1^2-x_2^2-x_3^2\\
        2(x_1x_2-x_3)\\
        2(x_1x_3+x_2)
    \end{pmatrix}
}
modulo orientation-preserving conformal transformations and modulo $\ker(\nabla\times)$.
\end{corollary}

Theorem~\ref{main_thm} settles an open problem going back at least to Rivi\`ere's 1998 paper~\cite{Riviere98CAG}. More precisely, Rivi\`ere proved that Killing forms are global minimizers of
\eq{\label{J_q}
  J_q(\alpha) \coloneqq \frac{ \Big(\int_{\S^n} \abs{\curl\alpha}^{q
  } \,\rdV \Big)^{\frac 2 q
  } }{ \int_{\S^n} \< \curl\alpha, \alpha\> \,\rdV},
}
for $q\ge 2$. He also showed that they are strict local minimizers (in the $C^1$ topology) for $q\in\bigl(\frac{2n}{n+1},2\bigr)$, and local minimizers for $q=\frac{2n}{n+1}$. The remaining question was whether Killing forms are global minimizers throughout the full range $q\in\bigl[\frac{2n}{n+1},2\bigr)$.

By H\"older's inequality, Theorem~\ref{main_thm} yields an affirmative answer for all $q\in\bigl[\frac{2n}{n+1},2\bigr)$, together with the sharp estimate
\eq{ J_q(\alpha) \ge \frac{n+1}{2}\,\omega_n^{\frac 2 q-1}.}
This global minimality also yields a quantitative stability result, in the spirit of Figalli--Zhang~\cite{Figalli_Zhang_20}; see \cite{WaZh26}.

In this paper, we do not address regularity issues.
Each result holds in suitable Sobolev spaces, which can be identified without difficulty.

As the first application, we determine the optimal metrics for the first positive curl eigenvalue in a conformal class on the sphere.

\begin{theorem}[Optimal metrics for the curl eigenvalue]\label{thm:optimal-metrics-intro}
Let $(\S^n,g_{{\rm st}})$ be the standard round sphere, and let
\[
    \mu([g_{{\rm st}}])\coloneqq \inf_{\tilde g\in[g_{{\rm st}}]} \lambda_1^+(\tilde g)\,\rVol(\tilde g)^{\frac{1}{n}},
\]
where $\lambda_1^+(\tilde g)$ denotes the first positive eigenvalue of the curl operator with respect to $\tilde g$. Then the round metric is the unique optimizer for $\mu([g_{{\rm st}}])$.
\end{theorem}

This is a Hersch-type (or El Soufi--Ilias type) theorem for the curl operator and is in fact equivalent to Theorem~\ref{main_thm}; see, for example, \cite{EGP25}. We give a detailed explanation and a proof in Section~\ref{sec0}. Theorem~\ref{thm:optimal-metrics-intro} is also a key ingredient in the proof of Theorem~\ref{thm:magnetic} below.

As the second application, we obtain an optimality result for the $3$-energy in the Hopf homotopy class.

\begin{theorem}[Hopf map minimizes the $3$-energy]\label{thm:hopf-3energy-intro}
The minimizers of the $3$-energy
\[
    \int_{\S^3}\abs{\rd u}^3,\quad u:\S^3\to\S^2,
\]
in the homotopy class of the Hopf map $\pi$ are precisely $\pi\circ\Phi$, where $\Phi\in{\rm Conf}^+(\S^3)$ is an orientation-preserving conformal transformation.
\end{theorem}

Using the results mentioned above, Rivi\`ere~\cite{Riviere98CAG} proved that the Hopf map $\pi$ minimizes the $p$-energy
\eq{
\int_{\S^3} \abs{\rd u}^p, \quad u:\S^3 \to \S^2,
}
in its homotopy class for $p\ge 4$, and that it is a local minimizer for $p\ge 3$. He conjectured that $\pi$ is a global minimizer of the $3$-energy, and provided evidence for this among the class of symmetric fibrations introduced in \cite{Riviere98CAG}. 
Since the $3$-energy is conformally invariant in dimension $3$, the conjecture is reminiscent of the Willmore conjecture proved by Marques--Neves~\cite{MarquesNevesWillmore2014}. Motivated by this analogy, Rivi\`ere~\cite{Riviere23} attempted to apply the min--max method of Marques--Neves to the $3$-energy and classified harmonic maps from $\S^3$ to $\S^2$ with index 4.

It was already observed in \cite{Riviere98CAG} that Theorem~\ref{main_thm} implies Theorem~\ref{thm:hopf-3energy-intro}.
In Section~\ref{app:3-energy} we revisit Rivi\`ere's approach and reduce the proof of Theorem~\ref{thm:hopf-3energy-intro} to Theorem~\ref{main_thm}.

As the third application, we settle the global minimality problem for the Hopf map in the Faddeev--Skyrme model.

\begin{theorem}[Faddeev--Skyrme minimizers in the Hopf class]\label{thm:fs-intro}
For any $\rho\le \sqrt{2}$, the minimizers of the Faddeev--Skyrme energy
\eq{\label{FS}
    \mathcal{FS}_\rho(u) \coloneqq \int_{\S^3} \abs{\rd u}^2 + \frac{1}{4\rho^2}\int_{\S^3} \abs{\rd u\w\rd u}^2, \quad u:\S^3\to \S^2,
}
in the homotopy class of the Hopf map $\pi$ are precisely $\pi\circ R$, where $R\in{\rm SO}(4)$.
\end{theorem}

It is known that the Hopf map $\pi$ is unstable for $\rho>\sqrt{2}$ \cite{Ward99}, whereas it is linearly stable for $0<\rho\le \sqrt{2}$; see \cite{Isobe08,SpeightSvensson07}. It is therefore natural to conjecture that, throughout the stable regime, $\pi$ is the unique minimizer of \eqref{FS} within its homotopy class, up to isometries of $\S^3$. Very recently, Guerra--Lamy--Zemas~\cite{GLZ26} verified this conjecture for $\rho\le 1$. The implication ``Theorem~\ref{main_thm} $\Rightarrow$ Theorem~\ref{thm:fs-intro}'' was also known to experts.

As the fourth application, we establish a sharp lower bound on the magnetic field in the presence of a Dirac zero mode.

\begin{theorem}\label{thm:magnetic}
Let $\varphi$ be a non-trivial zero mode on $\S^3$, i.e. a non-zero spinor satisfying
\eq{\label{eq1.0}
    \D\varphi = iA\cdot\varphi,
}
where $\D$ denotes the Dirac operator, $i$ is the imaginary unit, and $A$ is a vector field. Then
\eq{\label{eq:55}
    \norm{\curl A}_{\frac{3}{2}} \geq 3\,\omega_3^{\frac{2}{3}},
}
with equality if and only if, modulo conformal and gauge transformations, $\varphi$ is a Killing spinor and $A$ is a (real) multiple of the Reeb field associated with $\varphi$.
\end{theorem}

The zero mode equation \eqref{eq1.0} arises naturally in quantum electrodynamics (QED), where it describes zero-energy states of charged fermions coupled to an external magnetic field.
From a geometric viewpoint, \eqref{eq1.0} is a Dirac equation with vector potential $A$. The associated magnetic field is $\curl A$, which can be interpreted as the curvature of the connection $\rd+iA$.
Frank--Loss established $\norm{\curl A}_{3/2}\ge \tfrac{3}{2}\,\omega_3^{2/3}$ and raised the question ``What magnetic fields support a zero mode?" in \cite{FL1}.
Theorem~\ref{thm:magnetic} answers this question by proving the sharp lower bound $\norm{\curl A}_{3/2} \ge 3\,\omega_3^{2/3}$ and providing a complete characterization of the equality case. In particular, it yields a sharp criterion for the existence of zero modes. Another sharp criterion was proved by Frank--Loss in \cite{Frank_Loss_2024}.
We refer to \cites{Reuss25,WZ25b,Z26} for further results in this direction. For precise regularity assumptions on $\varphi$ and $\curl A$, see \cite{FL1}. The same question was also asked for $n\ge 3$, and we expect that our strategy extends to that setting.

For further applications in mathematical physics, see Appendix~B.

\medskip

\noindent\textit{Idea of the proof of Theorem~\ref{main_thm}.}
Let $\alpha$ be a minimizer for the quotient $J(\alpha)$. By combining the Euler--Lagrange equation with the scaling invariance of $J$, one reduces to the normalized nonlinear eigenvalue problem
\eq{\label{eq:main}
    \curl\alpha=\frac{n+1}{2}\abs{\alpha}^{\frac{2}{n-1}}\alpha.}
Writing $\alpha=f\beta$, where $f=\abs{\alpha}$ and $\abs{\beta}\equiv 1$, one decomposes the covariant derivative $\nabla\beta$ into three explicitly constructed, pointwise orthogonal components,
\[
    \nabla\beta=P+Q+S
\]
(Subsection~\ref{subsec:decomp-nablabeta}). This decomposition isolates the first-order contributions arising from $\nabla(\log f)$ from the zeroth-order terms depending only on $f$, and it yields a sharp Kato-type inequality which bounds $\abs{\nabla\alpha}^2$ from below in terms of $\abs{\nabla f}^2$ and $f^{\frac{2(n+1)}{n-1}}$ (Subsection~\ref{sec4.2}).

Combining the resulting pointwise estimate with the Bochner--Weitzenb\"ock formula on $\S^n$ and the Euler--Lagrange equation leads to the key integral inequality for $f$:
\eq{\label{eq1}
    \frac{(n-1)^2}{4} \int_{\S^n} f^{\frac{2(n+1)}{n-1}} \geq \int_{\S^n}\abs{\nabla f}^2 + \frac{(n-1)^2}{4}\int_{\S^n} f^2.
}
Accordingly, the proof reduces to a sharp scalar inequality, namely the spherical Gagliardo--Nirenberg inequality (Theorem~\ref{thm:GN-Sn}):
\eq{\label{eq2}
    \frac{(n-1)^2}{4}\omega_n^{\frac{2}{n}} \int_{\S^n} \abs{u}^{\frac{2(n+1)}{n-1}}
    \leq \Big( \int_{\S^n} \abs{u}^{\frac{2n}{n-1}} \Big)^{\frac{2}{n}} \int_{\S^n} \Big( \abs{\nabla u}^2 + \frac{(n-1)^2}{4}\abs{u}^2 \Big),
}
valid for every $u$ such that $\nabla u\in L^2$ and $u\in L^{\frac{2(n+1)}{n-1}}$. Combining \eqref{eq2} with \eqref{eq1} yields the sharp estimate $\int_{\S^n}f^{\frac{2n}{n-1}}\ge \omega_n$, which is equivalent to the desired curl--Sobolev inequality.

We briefly indicate the proof of \eqref{eq2}. The starting point is the sharp Euclidean Gagliardo--Nirenberg inequality of Del~Pino--Dolbeault~\cite{DelPino02}; see \eqref{eq:GN-Rn}. Via stereographic projection $\S^n\setminus\{N\}\to\R^n$, one transfers a function $f$ on $\S^n$ to a function $u$ on $\R^n$ with conformal weight $\tfrac{n-1}{2}$,
\eq{
    u(x) = \Big(\frac{2}{1+\abs{x}^2}\Big)^{\frac{n-1}{2}}f(\Psi(x)).
}
Under this correspondence, \eqref{eq2} becomes the weighted inequality \eqref{eq:GN-weighted}. One then applies the Euclidean sharp inequality to the Kelvin transform $\tilde u(x)=\abs{x}^{-(n-1)}u(x/\abs{x}^2)$, thereby obtaining a weighted sharp inequality \eqref{eq:7} for $u$. Averaging \eqref{eq:7} with \eqref{eq:GN-Rn} yields \eqref{eq:GN-weighted}. The invariance of the $L^{\frac{2n}{n-1}}$-norm under these conformal transformations is a crucial input.

Finally, equality in the curl--Sobolev inequality forces simultaneous equality in both the Kato-type inequality and the spherical Gagliardo--Nirenberg inequality of Theorem~\ref{thm:GN-Sn}. It follows that, modulo conformal transformations, $f$ is constant and the remainder term $S$ vanishes; consequently, $\alpha$ is a positive Killing form.

The method also yields alternative proofs of the spinorial Sobolev inequality~\cite{A03} and of the sharp lower bound for $\norm{A}_3$ in \cite{Frank_Loss_2024}, by using the subcritical Sobolev inequality \eqref{eq2} instead of the critical inequality \eqref{eq:S}. We omit the details.

As mentioned above, Applications~1--3 are either equivalent to Theorem~\ref{main_thm} (Application~1) or follow immediately from it (Applications~2 and~3). The proof of Application~4 requires additional input. While it follows a similar strategy, it also exploits a phenomenon specific to dimension~$3$ (Lemma~\ref{lem8.2}) in order to derive the Kato-type identity \eqref{eq:62}. One then combines the Schr\"odinger--Lichnerowicz formula, the critical Sobolev inequality \eqref{eq:74}, and Theorem~\ref{thm:optimal-metrics-intro} to conclude. In particular, the proof of Theorem~\ref{thm:magnetic} uses both the critical inequality \eqref{eq:74} and the subcritical inequality \eqref{eq2}.

Killing vector fields (or, more generally, Killing forms) need not minimize every conformally invariant functional. For instance, there exists a further Sobolev-type inequality associated with the curl operator on differential forms: there is a constant $S_2>0$ such that
\[
I(\alpha)=\frac{\big(\int|\curl\alpha|^{\frac{2n}{n+1}}\,\rdV\big)^{\frac{n+1}{n}}}{\inf_{\phi}\big(\int|\alpha-\rd\phi|^{\frac{2n}{n-1}}\,\rdV\big)^{\frac{n-1}{n}}}\ge S_2,
\]
for $\alpha\in \Omega^{\frac{n-1}2}(\S^n)$. The functional $I$ shares many features with $J$: it is conformally invariant, it admits minimizers \cite{FL22}, and its distinguished critical points are Killing forms. It is therefore natural to conjecture that Killing forms minimize $I$. If this were the case, then $S_2=\frac{(n+1)^2}{4}\,\omega_n^{\frac2n}$, as conjectured in \cite{Frank_Loss_2024}; see also \cites{FL22,FL1,MS25}. In contrast to Theorem~\ref{main_thm}, this conjecture is false; see \cite{WaZh26}. Determining the optimal constant $S_2$ remains open.

We conclude the introduction by formulating a Yamabe-type problem for differential forms. Let $n\equiv 3\pmod 4$ and let $(M,g)$ be a closed $n$-dimensional manifold. For an $\frac{n-1}{2}$-form $\alpha$, define
\eq{ J(\alpha) \coloneqq \frac{ \Big(\int_{M} \abs{\curl\alpha}^{\frac{2n}{n+1}} \,\rdV \Big)^{\frac{n+1}{n}} }{ \int_{M} \< \curl\alpha, \alpha\> \,\rdV},}
where $\rdV=\rdV(g)$ is the volume form of $g$. One may ask whether $J$ admits critical points and, in particular, whether it admits a minimizer.
Define the Yamabe-type constant
\eq{
    Y_{\frac {n-1}2}(M,[g]) :=\inf _{\int _M  \< \curl \alpha, \alpha\>>0}J(\alpha),
}
which is conformally invariant. A standard cut-off argument yields
\eq{\label{Y}
    Y_{\frac {n-1} 2}(M,[g]) \le Y_{\frac {n-1}2}(\S^n).
}
Moreover, blow-up analysis, combined with Theorem~\ref{main_thm}, shows that $Y_{\frac {n-1} 2}(M,[g])$ is achieved whenever the inequality in \eqref{Y} is strict. The remaining issue in this Yamabe problem is therefore the following:

\smallskip

{\it Does the strict inequality in \eqref{Y} hold for every manifold that is not conformally equivalent to the standard round sphere?}

\medskip

\noindent\textit{Organization of the paper.} Section~\ref{sec:prelim} collects notation and background material. In Section~\ref{sec:3} we prove the sharp spherical Gagliardo--Nirenberg inequality, and in Section~\ref{sec:main} we combine these ingredients to complete the proof of Theorem~\ref{main_thm}. The final part of the paper is devoted to applications, including optimal metrics for the first positive curl eigenvalue (Section~\ref{sec0}), the minimality of the Hopf map for the $3$-energy (Section~\ref{app:3-energy}), the global minimality of the Hopf map in the Faddeev--Skyrme model (Section~\ref{sec:app3}), and the sharp lower bound of the magnetic field of a zero mode (Section~\ref{sec:app4}). In appendix~A, we relate Theorem~\ref{thm:GN-Sn} to a weighted Yamabe problem. 
Further applications concerning sharp lower bounds for the original Faddeev--Skyrme model (fields $\R^3\to\S^2$) and several of its conformally invariant variants are collected in Appendix~B. 
A regularity theorem for weak solutions of \eqref{eq:main} is proved in Appendix C.

\medskip

\section{Preliminaries}\label{sec:prelim}

Throughout, we assume that $n\equiv 3\pmod 4$, so that the curl operator is self-adjoint.
\subsection{Exterior calculus}

Let $(M,g)$ be an oriented Riemannian manifold of dimension $n\equiv 3\pmod 4$. For each $0\le p\le n$, we denote by $\Omega^p$ the space of smooth $p$-forms and equip it with the $L^2$ inner product induced by $g$. The exterior derivative is denoted by $\rd:\Omega^p\to\Omega^{p+1}$, with the convention $\Omega^{n+1}=0$. We fix the orientation determined by the volume form $\rdV$. The Hodge star operator $*:\Omega^p\to\Omega^{n-p}$ is characterized by
\eq{\label{eq:def_Hodge_star}
    \alpha\w*\beta = \<\alpha, \beta\>\,\rdV.
}
The codifferential, i.e. the $L^2$-adjoint of $\rd$, is
\eq{
    \rd^*\coloneqq (-1)^{n(p+1)+1}*\rd*:\Omega^p\ra\Omega^{p-1}.
}

Since $n\equiv 3\pmod 4$, one has $\rd^*=-*\rd*$ on $\Omega^{\frac{n-1}{2}}$. We recall the standard identities
\eq{
    *^2 = (-1)^{p(n-p)}\,{\rm id},\qquad \rd^2 = (\rd^*)^2=0.
}
In particular, throughout this paper $*^2={\rm id}$.

The Hodge Laplacian is defined by ${\Delta \coloneqq \rd\rd^* + \rd^*\rd}$
and the rough Laplacian ${\nabla^*\nabla \coloneqq -\rtr\nabla^2}$.
These operators are related by the Weitzenb\"ock formula
\eq{
    \Delta = \nabla^*\nabla + \sum_{i,j} e^j \w e_i \ip R(e_i, e_j),
}
where $\{e_i\}$ is a local orthonormal frame, $\{e^i\}$ is the dual $1$-form coframe, $\ip$ is the interior multiplication, and $R$ is the curvature tensor. On the standard round sphere $\SS^n$, acting on $p$-forms, this reduces to
\eq{\label{sphere_Weitzenbock}
    \Delta = \nabla^*\nabla + p(n-p).
}

We record several identities that will be used repeatedly.

\begin{proposition}\label{basic_form}
Let $(M,g)$ be an oriented closed Riemannian manifold of dimension $n\equiv 3\pmod 4$. For any $\alpha,\beta\in\Omega^{\frac{n-1}{2}}$ and $X\in\Gamma(TM)$, the following identities hold:
\begin{enumerate}
    \item $\abs{X\ip\alpha}^2+\abs{X^\flat\w\alpha}^2=\abs{X}^2\abs{\alpha}^2$;
    \item $X\ip*\!\alpha=-*(X^\flat\w\alpha)$;
    \item $\<X\ip*\!\alpha,\beta\>=-\<\alpha, X\ip*\beta\>$, and in particular $\<X\ip*\!\alpha,\alpha\>=0$;
    \item $\rd = \sum_i e^i\w\nabla_{e_i}$ and $\rd^*=-\sum_i e_i\ip\nabla_{e_i}$;
    \item $\sum_i e^i\w e_i\ip\alpha = \frac{n-1}{2}\alpha$ and $\sum_i e_i\ip e^i\w\alpha = \frac{n+1}{2}\alpha$;
    \item $e^i\w e_j\ip\alpha + e_j\ip e^i\w\alpha = \delta_{ij}\alpha$.
\end{enumerate}
More generally, for $\alpha\in\Omega^p$ and $\beta\in\Omega^{n-p-1}$ the above formulas remain valid, with the following modifications:
\begin{enumerate}
    \item[(2)] $X\ip*\!\alpha=(-1)^p*(X^\flat\w\alpha)$;
    \item[(3)] $\<X\ip*\!\alpha,\beta\>=(-1)^p\<\alpha, X\ip*\beta\>$;
    \item[(5)] $\sum_i e^i\w e_i\ip\alpha = p\alpha$, \quad $\sum_i e_i\ip e^i\w\alpha = (n-p)\alpha$.
\end{enumerate}
In particular,
\eq{\label{direct_dual}
    \<X\ip\alpha,\beta\> = \<\alpha,X^\flat\w\beta\>, \quad\forall \alpha\in\Omega^p,\ \beta\in\Omega^{p-1}.
}
\end{proposition}

\subsection{The curl operator}\label{sec2.2}

Let $n\equiv 3\pmod 4$ and set $p\coloneqq\frac{n-1}{2}$. We define the curl operator by
\eq{\label{eq:curl-def}
    \curl \coloneqq *\rd: \Omega^{\frac{n-1}{2}}\ra \Omega^{\frac{n-1}{2}}.
}
When $n=3$ and $\alpha$ is the $1$-form dual to a vector field $X$, definition \eqref{eq:curl-def} coincides with the usual vector-calculus curl.

A basic computation using $\<\alpha,\beta\>\rdV=\alpha\wedge*\beta$ and Stokes' theorem shows that on a closed manifold $M$ the curl operator is formally self-adjoint on $\Omega^{\frac{n-1}{2}}(M)$:
\eq{\label{eq:curl-selfadj}
    \int_M \< \curl\alpha,\beta\>\,\rdV
    = \int_M \< \alpha,\curl\beta\>\,\rdV.
}
In particular,
\eq{\label{eq:curl-energy}
    \int_M \< \curl\alpha,\alpha\>\,\rdV = \int_M \alpha\wedge \rd\alpha.
}

Let $\tilde g=\rho^2g$ for some smooth function $\rho>0$. The standard conformal covariance is
\eq{\label{conformal_change}
    *_{\tilde{g}} = \rho^{n-2p} *_{g}, \quad \<\cdot,\cdot\>_{\tilde{g}} = \rho^{-2p}\<\cdot,\cdot\>_g \quad\hbox{on}\ \Omega^p.
}
In particular, for any $\alpha\in\Omega^{\frac{n-1}{2}}$,
\eq{
    \curl_{\tilde{g}}\alpha = *_{\tilde{g}}\rd\alpha = \rho^{-1}\,\curl_g\alpha.
}
Together with \eqref{conformal_change} this yields
\eq{
    \abs{\curl_{\tilde{g}}\alpha}_{\tilde{g}}^{\frac{2n}{n+1}} = \rho^{-n}\abs{\curl_g\alpha}_g^{\frac{2n}{n+1}},\qquad
    \<\curl_{\tilde{g}}\alpha, \alpha\>_{\tilde{g}} = \rho^{-n}\<\curl_g\alpha,\alpha\>_g.
}
Consequently, the following two quantities
\eq{\label{functionals}
    \int \abs{\curl\alpha}^{\frac{2n}{n+1}}\,\rdV, \qquad \int \<\curl\alpha,\alpha\>\,\rdV
}
are conformally invariant, and hence so is $J(\alpha)$.

\subsection{Killing forms}
We recall the definition and basic properties of Killing forms.

\begin{definition}\label{def_Killing}
A co-closed $p$-form $\alpha\in\Omega^p$ is called a \emph{Killing $p$-form} if
\eq{\label{eq:104}
    \nabla_X\alpha = \frac{1}{p+1}X\ip\rd\alpha,\quad \forall X\in\Gamma(TM).
}
\end{definition}

In particular, an $\frac{n-1}{2}$-Killing form is defined by
\eq{
    \nabla_X\alpha = \frac{2}{n+1}X\ip\rd\alpha,\quad \forall X\in\Gamma(TM).
}
On $\S^n$, each $\frac{n-1}{2}$-Killing form $\alpha$ can be decomposed as
\eq{
    \alpha = \alpha_++\alpha_-,
}
where $\alpha_\pm$ is a first $\pm$-eigenform of $\curl$, i.e.
\eq{
    \curl\alpha_+=\frac{n+1}{2}\alpha_+, \quad \curl\alpha_-=-\frac{n+1}{2}\alpha_-.
}
We call $\alpha$ a positive (respectively negative) $\frac{n-1}{2}$-Killing form if $\alpha=\alpha_+$ (respectively $\alpha=\alpha_-$).

The following properties are important; analogous statements hold for negative $\frac{n-1}{2}$-Killing forms.

\begin{proposition}\label{Killing_properties}
Let $\alpha$ be a positive $\frac{n-1}{2}$-Killing form on $\S^n$. Then
\begin{enumerate}
    \item $\nabla_X\alpha = X\ip*\alpha$ for every vector field $X$;
    \item $\abs{\alpha}$ is constant.
\end{enumerate}
\end{proposition}

\begin{proof}
(1) By definition we have $\curl\alpha = \frac{n+1}{2}\alpha$. Combining with Definition~\ref{def_Killing}, we have
\eq{
    \nabla_X\alpha = \frac{2}{n+1}X\ip \rd\alpha = \frac{2}{n+1}X\ip *\curl\alpha = X\ip *\alpha.
}

(2) It follows from (1) and Proposition~\ref{basic_form}(3).
\end{proof}

\subsection{Sharp Sobolev inequality on \texorpdfstring{$\S^n$}{Sn}}

For comparison, we recall the sharp critical Sobolev inequality on $\S^n$:
\eq{\label{eq:S}
\int_{\S^n} \Big( \abs{\nabla u}^2+\frac{n(n-2)}4 u^2 \Big) \ge \frac{n(n-2)}{4}\,\omega_n^{\frac{2}{n}}\left( \int_{\S^n}\abs{u}^{\frac{2n}{n-2}}\right)^{\frac{n-2}{n}}.
}
Our main result may be regarded as a form-analogue of \eqref{eq:S}. It is used also in the proof of Theorem \ref{thm:magnetic} in  section \ref{sec:app4}.

\section{A spherical Gagliardo--Nirenberg inequality on \texorpdfstring{$\S^n$}{Sn}}\label{sec:3}

In this section we prove a sharp Gagliardo--Nirenberg inequality on $\S^n$. We begin by recalling the sharp Euclidean Gagliardo--Nirenberg inequality of Del~Pino--Dolbeault~\cite{DelPino02}.

\begin{theorem}[{\cite[Theorem~1]{DelPino02}}]\label{thm:DelPinoGN}
Let $n\ge 2$ and set $p=\frac{n+1}{n-1}$. Then for all
\eq{
    u\in\mathcal{D}_p(\R^n) \coloneqq\{ u\in L^{p+1}(\R^n)\cap L^{2p}(\R^n)\,|\, \nabla u\in L^2(\R^n) \},
}
we have
\eq{\label{eq:GN-Rn}
    \frac{(n-1)^2}{4}\omega_n^{\frac{2}{n}} \int_{\R^n} \abs{u}^{\frac{2(n+1)}{n-1}}
    \leq \Big( \int_{\R^n} \abs{u}^{\frac{2n}{n-1}} \Big)^{\frac{2}{n}} \int_{\R^n} \abs{\nabla u}^2.
}
Moreover, equality holds if and only if
\eq{\label{eq:6}
    u(x) = \frac{c}{(\lambda^2 + \abs{x-x_0}^2)^{\frac{n-1}{2}}}
}
for some $c\in\R$, $\lambda>0$ and $x_0\in\R^n$.
\end{theorem}

Inequality \eqref{eq:GN-Rn} belongs to a family of sharp interpolation inequalities established by Del~Pino--Dolbeault, connecting the sharp Sobolev inequality with the logarithmic Sobolev inequality; see \cite{DelPino02,Case15} and the appendix.

We now turn to the corresponding sharp inequality on the sphere.

\begin{theorem}\label{thm:GN-Sn}
Let $n\ge 2$. For every smooth function $f$ on $\S^n$ we have
\eq{\label{eq:GN-Sn}
    \frac{(n-1)^2}{4}\omega_n^{\frac{2}{n}} \int_{\S^n} \abs{f}^{\frac{2(n+1)}{n-1}}
    \leq \Big( \int_{\S^n} \abs{f}^{\frac{2n}{n-1}} \Big)^{\frac{2}{n}} \int_{\S^n} \Big( \abs{\nabla f}^2 + \frac{(n-1)^2}{4}\abs{f}^2 \Big).
}
Moreover, equality holds if and only if $f$ is constant up to conformal transformations, i.e. there exists a conformal diffeomorphism $\Psi:\S^n\to \S^n$ with $\Psi^*g=\rho^2g$ such that $f=C\rho^{\frac{n-1}{2}}$ for some constant $C$.
\end{theorem}

\begin{proof}
We first consider the case $n\ge 3$.

\smallskip

\noindent\emph{Step 1 (Kelvin transform and averaging).} Let $u$ be obtained from a function on $\S^n$ by an inverse transformation of 
\eqref{eq:44} below. Consider the Kelvin transform
\eq{
    \tilde u(x)\coloneqq \abs{x}^{-(n-1)}u\Big(\frac{x}{\abs{x}^2}\Big).
}
It follows that
\eq{
    \int_{\R^n} \abs{\tilde{u}}^{\frac{2n}{n-1}} = \int_{\R^n} \abs{u}^{\frac{2n}{n-1}}, \qquad
    \int_{\R^n} \abs{\tilde{u}}^{\frac{2(n+1)}{n-1}} = \int_{\R^n} \abs{x}^2u^{\frac{2(n+1)}{n-1}}.
}
Moreover, set
\eq{
    y\coloneqq I(x) \coloneqq \frac{x}{\abs{x}^2}, \quad \rd x = \abs{y}^{-2n} \rd y.
}
Note that
\eq{
    DI_x = \abs{x}^{-2}(\rid-2\frac{x}{\abs{x}}\otimes\frac{x}{\abs{x}}) \eqcolon \abs{x}^{-2}R_x,
}
where $R_x$ is a reflection, with $R_x(x)=-x$. Hence
\eq{
    \nabla\tilde{u}(x) = -(n-1)\abs{x}^{-(n+1)}xu(y) + \abs{x}^{-(n-1)}(DI_x)\nabla u(y) = \abs{x}^{-(n+1)}(-(n-1)xu(y)+ R_x\nabla u(y)),
}
and then
\eq{
    \abs{\nabla\tilde{u}}^2(x) = \abs{x}^{-2(n+1)} \{ (n-1)^2\abs{x}^2u(y)^2 + \abs{\nabla u(y)}^2 + 2(n-1)u(y)\<x,\nabla u(y)\> \}.
}
Therefore,
\eq{
    \int_{\R^n} \abs{\nabla\tilde{u}}^2(x) \,\rd x &= \int_{\R^n} \abs{y}^2 \Big( \abs{\nabla u}^2 + (n-1)^2\abs{y}^{-2}u^2 + 2(n-1)u\abs{y}^{-2}\<y,\nabla u\> \Big) \rd y \\
    &= \int_{\R^n} \Big( \abs{y}^2\abs{\nabla u}^2 + (n-1)^2u^2 + 2(n-1)u\<y,\nabla u\> \Big) \rd y.
}
where
\eq{\label{eq:26}
    2(n-1)\int_{\R^n} u\<y,\nabla u\> \rd y = (n-1)\int_{\R^n} \<y,\nabla(u^2)\> \rd y = -n(n-1)\int_{\R^n} u^2.
}
In integration by parts used in the previous equation, we need to take care about the decay rate of $u$ at infinity. This causes no problem, when $n\ge 3$.  
It follows
\eq{
    \int_{\R^n} \abs{\nabla\tilde{u}}^2(x) \,\rd x = \int_{\R^n} \Big( \abs{y}^2\abs{\nabla u}^2 - (n-1)u^2 \Big) \rd y,
}
that is,
\eq{
    \int_{\R^n} \abs{\nabla\tilde{u}}^2 = \int_{\R^n} \abs{x}^2\abs{\nabla u}^2 - (n-1)\int_{\R^n} u^2.
}
Applying \eqref{eq:GN-Rn} to $\tilde{u}$ gives
\eq{\label{eq:7}
    \frac{(n-1)^2}{4}\omega_n^{\frac{2}{n}} \int_{\R^n} \abs{x}^2u^{\frac{2(n+1)}{n-1}} \leq \Big( \int_{\R^n} \abs{u}^{\frac{2n}{n-1}} \Big)^{\frac{2}{n}} \Big( \int_{\R^n} \abs{x}^2\abs{\nabla u}^2 - (n-1)\int_{\R^n} u^2 \Big).
}
Taking the average of \eqref{eq:GN-Rn} and \eqref{eq:7} yields
\eq{\label{eq:GN-weighted}
    \frac{(n-1)^2}{4}\omega_n^{\frac{2}{n}} \int_{\R^n} \frac{1+\abs{x}^2}{2}\abs{u}^{\frac{2(n+1)}{n-1}} 
    \leq \Big( \int_{\R^n} \abs{u}^{\frac{2n}{n-1}} \Big)^{\frac{2}{n}} \Big( \int_{\R^n} \frac{1+\abs{x}^2}{2}\abs{\nabla u}^2 - \frac{n-1}{2}\int_{\R^n} u^2 \Big).
}

\smallskip
\noindent\emph{Step 2 (Pull-back to $\S^n$).}
We now pull $u$ back to $\S^n$ via stereographic projection. Write the inverse stereographic projection $\Psi: \R^n \to \S^n\backslash\{N\}$ by
\eq{
    \Psi(x) = \Big( \frac{2x}{1+\abs{x}^2}, \frac{\abs{x}^2-1}{1+\abs{x}^2} \Big).
}
Then
\eq{\label{eq:95}
    \Psi^*g_{{\rm st}} = \left(\frac{2}{1+\abs{x}^2}\right)^2g_{\R^n}.
}
Define the function $f$ on $\S^n$ by
\eq{\label{eq:44}
    u(x) = \Big(\frac{2}{1+\abs{x}^2}\Big)^{\frac{n-1}{2}}f(\Psi(x)).
}
Then
\eq{
    \int_{\S^n} \abs{f}^{\frac{2n}{n-1}}\,\rd V_{\S^n} = \int_{\R^n} \abs{u}^{\frac{2n}{n-1}}\,\rd x,
}
and
\eq{\label{eq:11}
    \int_{\S^n} \abs{f}^{\frac{2(n+1)}{n-1}}\,\rd V_{\S^n} = \int_{\R^n} \frac{1+\abs{x}^2}{2}\abs{u}^{\frac{2(n+1)}{n-1}}\,\rd x, \qquad \int_{\S^n} \abs{f}^2\,\rd V_{\S^n} = \int_{\R^n} \frac{2}{1+\abs{x}^2}\abs{u}^2\,\rd x.
}
Moreover, we have
\eq{
    \nabla\Big(\Big(\frac{1+\abs{x}^2}{2}\Big)^{\frac{n-1}{2}}u\Big) = \Big(\frac{1+\abs{x}^2}{2}\Big)^{\frac{n-3}{2}}\Big(\frac{1+\abs{x}^2}{2}\nabla u + \frac{n-1}{2}xu\Big),
}
hence
\eq{\label{eq:45}
    \int_{\S^n} \abs{\nabla f}^2 &= \int_{\R^n} \Big(\frac{2}{1+\abs{x}^2}\Big)^{n-2}\Abs{\nabla\Big(\Big(\frac{1+\abs{x}^2}{2}\Big)^{\frac{n-1}{2}}u\Big)}^2 \\
    &= \int_{\R^n} \Big(\frac{2}{1+\abs{x}^2}\Big)^{n-2}\Abs{\Big(\frac{1+\abs{x}^2}{2}\Big)^{\frac{n-3}{2}}\Big(\frac{1+\abs{x}^2}{2}\nabla u + \frac{n-1}{2}xu\Big)}^2 \\
    &= \int_{\R^n} \frac{1+\abs{x}^2}{2}\abs{\nabla u}^2 + \frac{(n-1)^2}{4}\int_{\R^n} \frac{2\abs{x}^2}{1+\abs{x}^2}u^2 + \frac{n-1}{2}\int_{\R^n} \< x,\nabla(u^2)\> \\
    &= \int_{\R^n} \frac{1+\abs{x}^2}{2}\abs{\nabla u}^2 + \frac{(n-1)^2}{4}\int_{\R^n} \frac{2\abs{x}^2}{1+\abs{x}^2}u^2 - \frac{n-1}{2}\int_{\R^n} \rdiv(x)u^2 \\
    &= \int_{\R^n} \frac{1+\abs{x}^2}{2}\abs{\nabla u}^2 + \frac{n-1}{2}\int_{\R^n} \Big( \frac{(n-1)\abs{x}^2}{1+\abs{x}^2} - n \Big)u^2\\
    &= \int_{\R^n} \frac{1+\abs{x}^2}{2}\abs{\nabla u}^2 - \frac{(n-1)^2}{4}\int_{\R^n} \frac{2}{1+\abs{x}^2}u^2 - \frac{n-1}{2}\int_{\R^n} u^2.
}
It follows
\eq{\label{eq:10}
    \int_{\S^n} \Big( \abs{\nabla f}^2 + \frac{(n-1)^2}{4}f^2 \Big) = \int_{\R^n} \frac{1+\abs{x}^2}{2}\abs{\nabla u}^2 - \frac{n-1}{2}\int_{\R^n} u^2.
}

Combining \eqref{eq:GN-weighted} with \eqref{eq:11} and \eqref{eq:10} yields
\eq{
    \frac{(n-1)^2}{4}\omega_n^{\frac{2}{n}} \int_{\S^n} \abs{f}^{\frac{2(n+1)}{n-1}} \leq \Big( \int_{\S^n} \abs{f}^{\frac{2n}{n-1}} \Big)^{\frac{2}{n}} \int_{\S^n} \Big( \abs{\nabla f}^2 + \frac{(n-1)^2}{4}f^2 \Big),
}
the desired inequality. The characterization of the equality case follows by pulling back \eqref{eq:6}. See Lemma~\ref{lem:extremals-conformal} below for details.

We now consider the case $n=2$. In this case, the integration by parts in \eqref{eq:26} produces a boundary contribution at infinity. After averaging, this yields an additional term $2\pi f(N)^2$ in the second factor on the right-hand side of \eqref{eq:GN-weighted}.

More precisely, the last identity in \eqref{eq:26} must be carried out on the ball $B_R\subset\R^2$:
\eq{
    \int_{B_R}\<y,\nabla(u^2)\>\rd y
    = -2\int_{B_R} u^2\,\rd y + \int_{\partial B_R} u^2\<y,\nu\>\,\rd\sigma
    = -2\int_{B_R} u^2\,\rd y + \int_{\partial B_R} Ru^2\,\rd\sigma.
}
Using \eqref{eq:44} and writing $y=R\theta$, we obtain
\eq{
    \int_{\partial B_R} Ru^2\,\rd\sigma
    = \frac{2R^2}{1+R^2}\int_{\S^1} f(\Psi(R\theta))^2\,\rd\theta
    \longrightarrow 4\pi f(N)^2 \quad\hbox{as } R\to\infty.
}
Thus
\eq{
    \int_{B_R}\<y,\nabla(u^2)\>\rd y = -2\int_{B_R} u^2\,\rd y + 4\pi f(N)^2 + o(1),
}
and letting $R\to\infty$ gives the desired correction term.
Consequently, \eqref{eq:7} becomes
\eq{
    \frac{1}{4}\omega_2 \int_{\R^2} \abs{x}^2u^6
    \leq \int_{\R^2} u^4 \cdot \Big( \int_{\R^2} \abs{x}^2\abs{\nabla u}^2 -\int_{\R^2} u^2 + 4\pi f(N)^2 \Big),
}
so after averaging we obtain
\eq{\label{eq:46}
    \frac{1}{4}\omega_2 \int_{\R^2} \frac{1+\abs{x}^2}{2}u^6
    \leq \int_{\R^2} u^4 \cdot \Big( \int_{\R^2} \frac{1+\abs{x}^2}{2}\abs{\nabla u}^2 - \frac{1}{2}\int_{\R^2} u^2 + 2\pi f(N)^2 \Big).
}
Similarly, integrating by parts in \eqref{eq:45} yields the same correction, so \eqref{eq:10} becomes
\eq{\label{eq:47}
    \int_{\S^2} \Big( \abs{\nabla f}^2 + \frac{1}{4}f^2 \Big)
    = \int_{\R^2} \frac{1+\abs{x}^2}{2}\abs{\nabla u}^2 - \frac{1}{2}\int_{\R^2} u^2 + 2\pi f(N)^2.
}
Combining \eqref{eq:46}, \eqref{eq:47}, and \eqref{eq:11}, we see that
\eq{
    \frac{1}{4}\omega_2 \int_{\S^2} f^6 \leq \int_{\S^2} f^4 \cdot \int_{\S^2} \Big( \abs{\nabla f}^2 + \frac{1}{4}f^2 \Big),
}
hence the conclusion still holds.
\end{proof}

\begin{lemma}[Extremals and conformal factors]\label{lem:extremals-conformal}
Let $n\ge 2$ and let $f\ge 0$ be smooth on $\S^n$. Define $u$ on $\R^n$ by \eqref{eq:44}. If $u$ is an extremal of the Del~Pino--Dolbeault inequality \eqref{eq:GN-Rn}, i.e.
\eq{\label{eq:94}
    u(x)=\frac{c}{(\lambda^2+\abs{x-x_0}^2)^{\frac{n-1}{2}}}
}
for some $c>0$, $\lambda>0$, and $x_0\in\R^n$, then there exists a conformal diffeomorphism $\Psi:\S^n\to\S^n$ with $\Psi^*g=\rho^2g$ such that
\eq{\label{eq:f-conformal-factor}
    f = C\,\rho^{\frac{n-1}{2}}
}
for some constant $C>0$.
\end{lemma}

\begin{proof}
Conformal transformations of $\S^n$ are parameterized by
\eq{
    \Psi_a(p) \coloneqq -a + \frac{1-\abs{a}^2}{\abs{p-a}^2}(p-a), \quad a\in\mathbb{B}^{n+1},\ p\in\S^n.
}
The pull-back metric $\Psi_a^*g_{{\rm st}}=\rho_a^2g_{{\rm st}}$ has the conformal factor
\eq{
    \rho_a(p) = \frac{1-\abs{a}^2}{\abs{p-a}^2} = \frac{1-\abs{a}^2}{1+\abs{a}^2-2\<a,p\>}, \quad p\in\S^n.
}
Using \eqref{eq:44} and \eqref{eq:94} we have for $p=\Psi(x)$
\eq{
    f(p) = \Big(\frac{1+\abs{x}^2}{2}\Big)^{\frac{n-1}{2}}u(x) = c(2\lambda)^{-\frac{n-1}{2}}\Big( \frac{\lambda(1+\abs{x}^2)}{\lambda^2+\abs{x-x_0}^2} \Big)^{\frac{n-1}{2}}.
}
Choose
\eq{
    a = \frac{ (2x_0,\lambda^2+\abs{x_0}^2-1) }{ \abs{x_0}^2+(\lambda+1)^2 } \in\mathbb{B}^{n+1},
}
then we have
\eq{
    \rho_a(p) = \frac{\lambda(1+\abs{x}^2)}{\lambda^2+\abs{x-x_0}^2}.
}
Hence
\eq{
    f(p) = c(2\lambda)^{-\frac{n-1}{2}}\rho_a(p)^{\frac{n-1}{2}}
    .
}
The claim follows.
\end{proof}
All  solutions on $\S^n$ can be expressed by 
\[
C\Big( \frac{1-\abs{a}^2}{1+\abs{a}^2-2\<a,p\>} \Big)^{\frac{n-1}{2}},
\qquad C>0,\quad a\in \mathbb{B}^{n+1}.\]

\section{Sharp curl--Sobolev inequality}\label{sec:main}

In this section we prove the sharp curl--Sobolev inequality, Theorem~\ref{main_thm}.

Note that
\eq{
    J(\alpha)=J(\alpha+\gamma), \quad\forall\gamma\in\ker(\curl\!)=\ker(\rd).
}
After choosing a suitable representative and rescaling, we may assume that $\alpha$ solves the Euler--Lagrange equation
\eq{\label{eq:EL-curl}
    \curl\alpha=\frac{n+1}{2}\abs{\alpha}^{\frac{2}{n-1}}\alpha.
}
(With this normalization, it is easy to see that a Killing form $\alpha_0$ with $\abs{\alpha_0}=1$ satisfies \eqref{eq:EL-curl}, since $\alpha_0$ is a $\curl$-eigenform with eigenvalue $\frac{n+1}{2}$.)
Then
\eq{
    J(\alpha) = \frac{n+1}{2}\Big(\int_{\S^n}\abs{\alpha}^{\frac{2n}{n-1}}\Big)^{\frac{1}{n}}.
}
Hence the goal is to prove 
\eq{\label{main_ineq}
    \int_{\S^n}\abs{\alpha}^{\frac{2n}{n-1}} \geq \omega_n.
}

We now use a result of Rivi\`ere
\cite[Proposition IV.1]{Riviere98CAG}, where he proved the existence of a minimizer of $J$ among all $\alpha \in \Omega^{\frac {n-1}2}$ with $\int_{\S^n} \< \curl \alpha , \alpha \> \rdV >0$. 
We assume from now on that $\alpha$ is a minimizer, which is in $W^{1,\frac{2n}{n+1}}(\S^n)$. We remark that
we only need to consider the points $x\in \S^n$ with $|\alpha(x)|>0$, see Remark \ref{rem_zero}. 
The regularity theorem about solutions of \eqref{eq:EL-curl} implies that $\alpha \in W^{1,2}(\S^n)\cap C^{1,\gamma}(\{\alpha\not =0\})$. For convenience of the reader, we sketch the proof of the regularity in appendix
\ref{app:C}.
Hence, the following computation is done on $\{\alpha\not =0\}$.

Set
\eq{
    f\coloneqq\abs{\alpha}, \quad \alpha=f\beta, \quad \abs{\beta}=1.
}
Then the Euler--Lagrange equation becomes
\eq{\label{eq:12}
    *\rd(f\beta) = \frac{n+1}{2}f^{\frac{n+1}{n-1}}\beta.
}

\begin{remark}\label{rem_zero}
On the open set $\{f>0\}$ the form $\beta$ is well-defined and smooth, and all calculations are performed on $\{f>0\}$. Since $f\in W^{1,2}(\S^n)$ (see Theorem~\ref{thmC.1}), both $\nabla\alpha$ and $\nabla f$ vanish almost everywhere on the zero set $\{f=0\}$. See also Remark \ref{rem4.8}. Therefore, the resulting integral identities extend to the whole sphere $\S^n$.
\end{remark}

\begin{lemma}\label{lem4.1}
We have
\begin{enumerate}
    \item $\rd\beta=\frac{n+1}{2}f^{\frac{2}{n-1}}*\beta - \rd(\log f)\w\beta$;
    \item $\rd^*\beta = \frac{n+1}{n-1}\nabla(\log f)\ip\beta$;
    \item $\rd^*\alpha=\frac{2}{n-1}\nabla f\ip\beta$.
\end{enumerate}
\end{lemma}
\begin{proof}
(1) Using \eqref{eq:12} we have
\eq{
    f\rd\beta + \rd f\w\beta = \frac{n+1}{2}f^{\frac{n+1}{n-1}}*\beta.
}
Dividing by $f$ gives the desired formula.

(2) Using \eqref{eq:12} we have
\eq{
    0 = \rd^*(f^{\frac{n+1}{n-1}}\beta) = f^{\frac{n+1}{n-1}}\rd^*\beta - \nabla(f^{\frac{n+1}{n-1}})\ip\beta.
}
Solving for $\rd^*\beta$ gives the claim.

(3) By definition we have
\eq{
    \rd^*\alpha = \rd^*(f\beta) = f\rd^*\beta - \nabla f\ip\beta.
}
The claim follows from (2).
\end{proof}

\subsection{Decomposition of \texorpdfstring{$\nabla\beta$}{nabla beta}}
\label{subsec:decomp-nablabeta}

We view $\nabla\beta$ as a $T^*\S^n$-valued $\frac{n-1}{2}$-form, i.e. as the assignment $X\mapsto \nabla_X\beta$ for vector fields $X$. For each $X$ we set
\eq{
    P_X &\coloneqq \frac{2}{n-1}\Big(\rd(\log f)\w X\ip\beta - X^\flat\w \nabla(\log f)\ip\beta\Big), \\
    Q_X &\coloneqq f^{\frac{2}{n-1}}\,X\ip*\beta, \\
    S_X &\coloneqq \nabla_X\beta - P_X - Q_X.
}
Then
\eq{\label{eq:decomposition}
    \nabla\beta = P+Q+S.
}

Let us briefly explain the motivation for \eqref{eq:decomposition}. Any matrix $A$ has the orthogonal decomposition
\eq{
    A = \frac{1}{2}(A-A^T) + \frac{1}{n}{\rm tr}(A)\,\rid + \Big(\frac{1}{2}(A+A^T)-\frac{1}{n}{\rm tr}(A)\,\rid\Big),
}
corresponding to the anti-symmetric part, the trace part, and the symmetric trace-free part. Analogously, a $T^*\S^n$-valued form $T$ has the anti-symmetric trace
\eq{
    e^i\w T_{e_i}
}
and the trace part
\eq{
    e_i\ip T_{e_i},
}
where $\{e_i\}$ is a local orthonormal frame and $\{e^i\}$ is the dual coframe.

For $T=\nabla\beta$, we have already seen from Lemma~\ref{lem4.1} that the anti-symmetric part is
\eq{
    e^i\w\nabla_{e_i}\beta = \rd\beta = \frac{n+1}{2}f^{\frac{2}{n-1}}*\beta - \rd(\log f)\w\beta,
}
and the trace part is
\eq{
    e_i\ip \nabla_{e_i}\beta = -\rd^*\beta = -\frac{n+1}{n-1}\nabla(\log f)\ip\beta,
}
where we used Proposition~\ref{basic_form}(4). The remaining component is what we call $S$.

To obtain the sharp Sobolev-type estimate \eqref{eq:27}, we separate the first-order terms (involving $\rd(\log f)$) from the zeroth-order terms (involving only $f$). This leads to the definitions of $P$ and $Q$. By construction, $P$ and $Q$ achieve this separation, and Lemma~\ref{lem4.2} below shows that
\eq{
    e^i\w\nabla_{e_i}\beta = e^i\w P_{e_i} + e^i\w Q_{e_i},
}
and
\eq{
    e_i\ip \nabla_{e_i}\beta = e_i\ip P_{e_i} + e_i\ip Q_{e_i}.
}

\begin{lemma}\label{lem4.2}
Let $\{e_i\}$ be a local frame. Then
\begin{enumerate}
    \item $e^i\w P_{e_i}=-\rd(\log f)\w\beta$;
    \item $e^i\w Q_{e_i}=\frac{n+1}{2}f^{\frac{2}{n-1}}*\beta$;
    \item $e^i\w S_{e_i}=0$;
    \item $e_i\ip P_{e_i}=-\frac{n+1}{n-1}\nabla(\log f)\ip\beta$;
    \item $e_i\ip Q_{e_i}=0$;
    \item $e_i\ip S_{e_i}=0$.
\end{enumerate}
\end{lemma}
\begin{proof}
For (1),
\eq{
    e^i\w P_{e_i} &= \frac{2}{n-1}e^i\w\Big(\rd(\log f)\w e_i\ip\beta - e^i\w \nabla(\log f)\ip\beta\Big) \\
    &= -\frac{2}{n-1}\rd(\log f)\w e^i\w e_i\ip\beta \\
    &= -\rd(\log f)\w\beta,
}
where we used Proposition~\ref{basic_form}(5).

For (2),
\eq{
    e^i\w Q_{e_i} = f^{\frac{2}{n-1}}e^i\w e_i\ip*\beta = \frac{n+1}{2}f^{\frac{2}{n-1}}*\beta,
}
where we used Proposition~\ref{basic_form}(5).

For (3),
\eq{
    e^i\w S_{e_i} &= e^i\w ( \nabla_{e_i}\beta - P_{e_i} -Q_{e_i} ) \\
    &= \rd\beta - e^i\w P_{e_i} - e^i\w Q_{e_i} \\
    &= \frac{n+1}{2}f^{\frac{2}{n-1}}*\beta - \rd(\log f)\w\beta + \rd(\log f)\w\beta - \frac{n+1}{2}f^{\frac{2}{n-1}}*\beta \\
    &= 0,
}
where we used Proposition~\ref{basic_form}(4) and Lemma~\ref{lem4.1}(1).

For (4),
\eq{
    \frac{n-1}{2} e_i\ip P_{e_i} &= e_i\ip \rd(\log f)\w e_i\ip\beta - e_i\ip e^i\w\nabla(\log f)\ip\beta \\
    &= e_i(\log f)e_i\ip\beta - \rd(\log f)\w e_i\ip e_i\ip\beta - \frac{n+3}{2}\nabla(\log f)\ip\beta \\
    &= -\frac{n+1}{2} \nabla(\log f)\ip\beta,
}
where we used Proposition~\ref{basic_form}(6).

For (5),
\eq{
    e_i\ip Q_{e_i} = f^{\frac{2}{n-1}} e_i\ip e_i\ip*\beta = 0.
}

For (6),
\eq{
    e_i\ip S_{e_i} = e_i\ip\nabla_{e_i}\beta - e_i\ip P_{e_i} - e_i\ip Q_{e_i} = -\rd^*\beta - e_i\ip P_{e_i} = 0,
}
where we used (4) and Lemma~\ref{lem4.1}(2).

\end{proof}

As expected, we show that the decomposition \eqref{eq:decomposition} is orthogonal.

\begin{lemma}\label{lem4.3}
The tensors $P,Q,S$ are pairwise orthogonal with respect to the pointwise inner product, i.e.
\eq{
    \< P,Q\>=\< P,S\>=\< Q,S\>=0.
}
\end{lemma}
\begin{proof}
Let $\{e_i\}$ be a local frame. First, using Lemma~\ref{lem4.2}(1),
\eq{
    \<P,Q\> &= \<P_{e_i},Q_{e_i}\> = f^{\frac{2}{n-1}}\<P_{e_i}, e_i\ip*\beta\> = f^{\frac{2}{n-1}}\<e^i\w P_{e_i}, *\beta\> \\
    &= -f^{\frac{2}{n-1}}\<\rd(\log f)\w\beta,*\beta\> = 0,
}
where we used \eqref{eq:def_Hodge_star}, \eqref{direct_dual}, and the fact that $\deg\beta=\frac{n-1}{2}$ is odd, hence $\beta\w\beta=0$.

Second, using Lemma~\ref{lem4.2}(3),
\eq{
    \<S,Q\> = \<S_{e_i},Q_{e_i}\> = f^{\frac{2}{n-1}}\<S_{e_i}, e_i\ip*\beta\> = f^{\frac{2}{n-1}}\<e^i\w S_{e_i}, *\beta\> = 0,
}
where we used \eqref{direct_dual}.

Third, we show that $\<P,S\>=0$. Note that
\eq{
    \frac{n-1}{2}\<P,S\> = \<P_{e_i},S_{e_i}\> = \< \rd(\log f)\w e_i\ip\beta, S_{e_i}\> - \<e^i\w\nabla(\log f)\ip\beta, S_{e_i}\>. 
}
The first term is
\eq{
    \< \rd(\log f)\w e_i\ip\beta, S_{e_i}\> &= \<e_i(\log f)\beta - e_i\ip\rd(\log f)\w\beta, S_{e_i}\> \\
    &= e_i(\log f)\<\beta,S_{e_i}\> - \<\rd(\log f)\w\beta, e^i\w S_{e_i}\> \\
    &= e_i(\log f)\<\beta,S_{e_i}\> \\
    &= e_i(\log f)\<\beta,\nabla_{e_i}\beta - P_{e_i} - Q_{e_i}\>,
}
where we used \eqref{direct_dual}, Proposition~\ref{basic_form}(6), and Lemma~\ref{lem4.2}(3). Since $\abs{\beta}=1$, we have
\eq{
    \<\beta,\nabla_{e_i}\beta\> = 0.
}
Moreover, 
\eq{
    \frac{n-1}{2}\<\beta,P_{e_i}\> &= \<\beta,\rd(\log f)\w e_i\ip\beta\> - \<\beta,e^i\w\nabla(\log f)\ip\beta\> \\
    &= \<\nabla(\log f)\ip\beta,e_i\ip\beta\> - \<e_i\ip\beta,\nabla(\log f)\ip\beta\> = 0,
}
where we used \eqref{direct_dual}; and
\eq{
    \<\beta,Q_{e_i}\> &= f^{\frac{2}{n-1}}\<\beta,e_i\ip*\beta\> = 0.
}
Hence the first term vanishes:
\eq{
    \< \rd(\log f)\w e_i\ip\beta, S_{e_i}\> = 0.
}
The second term is
\eq{
    \<e^i\w\nabla(\log f)\ip\beta, S_{e_i}\> = \< \nabla(\log f)\ip\beta, e_i\ip S_{e_i}\> = 0,
}
where we used Lemma~\ref{lem4.2}(6). Therefore
\eq{
    \<P,S\>=0.
}
We complete the proof.
\end{proof}

\begin{lemma}\label{lem4.4}
We have
\begin{enumerate}
    \item $\abs{P}^2 = \frac{2}{n-1}\abs{\nabla(\log f)}^2+\frac{4}{(n-1)^2}\abs{\nabla(\log f)\ip\beta}^2$;
    \item $\abs{Q}^2=\frac{n+1}{2}f^{\frac{4}{n-1}}$.
\end{enumerate}
\end{lemma}
\begin{proof}
(1) First,
\eq{
    \sum_{i=1}^n\abs{\rd(\log f)\w e_i\ip\beta}^2 &= \sum_{i=1}^n\abs{e_i(\log f)\beta - e_i\ip\rd(\log f)\w\beta}^2 \\
    &= \abs{\nabla(\log f)}^2 + \frac{n+1}{2}\abs{\rd(\log f)\w\beta}^2 - 2\<\beta,\nabla(\log f)\ip\rd(\log f)\w\beta\> \\
    &= \abs{\nabla(\log f)}^2 + \frac{n-3}{2}\abs{\rd(\log f)\w\beta}^2,
}
where we used \eqref{direct_dual} and Proposition~\ref{basic_form}(5).
Second,
\eq{
    \sum_{i=1}^n\abs{e^i\w\nabla(\log f)\ip\beta}^2 = \frac{n+3}{2}\abs{\nabla(\log f)\ip\beta}^2,
}
where we used \eqref{direct_dual} and Proposition~\ref{basic_form}(5).
Third,
\eq{
    \<\rd(\log f)\w e_i\ip\beta, e^i\w\nabla(\log f)\ip\beta\> &= \<e_i(\log f)\beta - e_i\ip\rd(\log f)\w\beta,e^i\w\nabla(\log f)\ip\beta\> \\
    &= \abs{\nabla(\log f)\ip\beta}^2,
}
where we used \eqref{direct_dual} and Proposition~\ref{basic_form}(5)(6).
Hence
\eq{
    \frac{(n-1)^2}{4}\abs{P}^2 &= \frac{(n-1)^2}{4}\<P_{e_i},P_{e_i}\> \\
    &= \sum_{i=1}^n\abs{\rd(\log f)\w e_i\ip\beta}^2 + \sum_{i=1}^n\abs{e^i\w\nabla(\log f)\ip\beta}^2 - 2\<\rd(\log f)\w e_i\ip\beta, e^i\w\nabla(\log f)\ip\beta\> \\
    &= \sum_{i=1}^n\abs{e_i(\log f)\beta - e_i\ip\rd(\log f)\w\beta}^2 + \frac{n+3}{2}\abs{\nabla(\log f)\ip\beta}^2 - 2\<e_i\ip\rd(\log f)\w e_i\ip\beta, \nabla(\log f)\ip\beta\> \\
    &= \abs{\nabla(\log f)}^2 + \frac{n-3}{2}\abs{\rd(\log f)\w\beta}^2 + \frac{n+3}{2}\abs{\nabla(\log f)\ip\beta}^2 - 2\abs{\nabla(\log f)\ip\beta}^2 \\
    &= \abs{\nabla(\log f)}^2 + \frac{n-3}{2}\Big(\abs{\nabla(\log f)}^2-\abs{\nabla(\log f)\ip\beta}^2\Big) + \frac{n-1}{2}\abs{\nabla(\log f)\ip\beta}^2 \\
    &= \frac{n-1}{2}\abs{\nabla(\log f)}^2 + \abs{\nabla(\log f)\ip\beta}^2,
}
where we used \eqref{direct_dual} and Proposition~\ref{basic_form}(1)(5)(6).
The claim follows.

(2) Using \eqref{direct_dual} and Proposition~\ref{basic_form}(5) we have
\eq{
    \abs{Q}^2 = f^{\frac{4}{n-1}}\sum_{i=1}^n\abs{e_i\ip*\beta}^2 = \frac{n+1}{2}f^{\frac{4}{n-1}},
}
the claim.
\end{proof}

\begin{corollary}\label{cor4.5}
We have
\eq{
    \abs{\nabla\beta}^2 = \frac{2}{n-1}\abs{\nabla(\log f)}^2+\frac{4}{(n-1)^2}\abs{\nabla(\log f)\ip\beta}^2 + \frac{n+1}{2}f^{\frac{4}{n-1}} + \abs{S}^2.
}
\end{corollary}
\begin{proof}
This is a consequence of Lemma~\ref{lem4.3} and Lemma~\ref{lem4.4}.
\end{proof}

\subsection{A Kato-type identity}\label{sec4.2}

Corollary \ref{cor4.5} gives a Kato-type identity. 
\begin{proposition}\label{prop4.6} 
We have
\eq{
    \abs{\nabla\alpha}^2 = \frac{n+1}{n-1}\abs{\nabla f}^2 + \frac{4}{(n-1)^2}\abs{\nabla f\ip\beta}^2 + \frac{n+1}{2}f^{\frac{2(n+1)}{n-1}} + f^2\abs{S}^2. 
}
In particular,
\eq{
    \label{eq:Kato2}\abs{\nabla\alpha}^2 \geq \frac{n+1}{n-1}\abs{\nabla f}^2 + \frac{4}{(n-1)^2}\abs{\nabla f\ip\beta}^2 + \frac{n+1}{2}f^{\frac{2(n+1)}{n-1}}.
}
\end{proposition}
\begin{proof}
Since
\eq{
    \nabla\alpha = \nabla(f\beta) = f\nabla\beta + \nabla f\otimes\beta,
}
we have
\eq{
    \abs{\nabla\alpha}^2 = f^2\abs{\nabla\beta}^2 + \abs{\nabla f}^2.
}
The claim then follows from Corollary~\ref{cor4.5}.
\end{proof}

\begin{remark}\label{rem4.8}
We remark that \eqref{eq:Kato2} can be written as
\eq{\label{Kato_ineq}
    \abs{\nabla\alpha}^2 \geq \frac{n+1}{n-1}\abs{\nabla \abs{\alpha}}^2 + \frac{4}{(n-1)^2}\abs{\nabla \abs{\alpha}\ip\frac{\alpha}{\abs{\alpha}}}^2 + \frac{n+1}{2}\abs{\alpha}^{\frac{2(n+1)}{n-1}},
}
a Kato-type inequality. Note that by Lemma \ref{lem4.1}
\eq{
\nabla |\alpha|\ip \frac {\alpha}{\abs{\alpha}}=\nabla f\ip \beta=\frac {n-1} 2 \rd^*\alpha,
}
and hence \eqref{Kato_ineq} holds on $\S^n$.
\end{remark}

Now we use the Bochner--Weitzenb\"ock formula to show 
\begin{lemma}\label{thm4.7}
We have
\eq{\label{eq:27}
    \frac{(n-1)^2}{4} \int_{\S^n} f^{\frac{2(n+1)}{n-1}} \geq \int_{\S^n}\abs{\nabla f}^2 + \frac{(n-1)^2}{4}\int_{\S^n} f^2.
}
\end{lemma}
\begin{proof}
Recall the Bochner--Weitzenb\"ock formula for $\frac{n-1}{2}$-forms on $\S^n$:
\eq{\label{eq:16}
    \int_{\S^n} \abs{\rd\alpha}^2 + \int_{\S^n} \abs{\rd^*\alpha}^2 = \int_{\S^n} \abs{\nabla\alpha}^2 + \frac{(n+1)(n-1)}{4}\int_{\S^n} \abs{\alpha}^2.
}
By the Euler--Lagrange equation \eqref{eq:EL-curl} we have
\eq{\label{eq:17}
    \int_{\S^n} \abs{\rd\alpha}^2 = \frac{(n+1)^2}{4} \int_{\S^n} f^{\frac{2(n+1)}{n-1}}.
}
By Lemma~\ref{lem4.1}(3) we have
\eq{\label{eq:18}
    \int_{\S^n} \abs{\rd^*\alpha}^2 = \frac{4}{(n-1)^2}\int_{\S^n} \abs{\nabla f\ip\beta}^2.
}
By Proposition~\ref{prop4.6} we have
\eq{\label{eq:19}
    \int_{\S^n} \abs{\nabla\alpha}^2 \geq \frac{n+1}{n-1}\int_{\S^n}\abs{\nabla f}^2 + \frac{4}{(n-1)^2}\int_{\S^n}\abs{\nabla f\ip\beta}^2 + \frac{n+1}{2}\int_{\S^n}f^{\frac{2(n+1)}{n-1}}.
}
The claim follows from \eqref{eq:16}, \eqref{eq:17}, \eqref{eq:18}, and \eqref{eq:19}.
\end{proof}

\subsection{Proof of the main theorem}
We now prove Theorem~\ref{main_thm}.

\medskip

\begin{proof}[Proof of Theorem \ref{main_thm}] 
By Theorem~\ref{thm:GN-Sn} and Lemma~\ref{thm4.7}, we obtain
\eq{\label{eq:29}
    \int_{\S^n} \abs{\alpha}^{\frac{2n}{n-1}} = \int_{\S^n} f^{\frac{2n}{n-1}} \geq \omega_n,
}
which is exactly \eqref{main_ineq} and hence proves the desired inequality.

We now characterize the equality case. We claim the following conformal covariance of $S$:
\eq{\label{eq:28}
    \abs{\alpha}_{\tilde{g}}=1, \quad \tilde{S}_X = fS_X, \quad \hbox{ for }\tilde{g}=f^{\frac{4}{n-1}}g, \quad g=g_{{\rm st}}.
}

\begin{proof}[ Proof of the Claim.]
Let $\rho\coloneqq f^{\frac{2}{n-1}}$, then $\tilde{g}=\rho^2g$. Using \eqref{conformal_change} we have
\eq{
    \abs{\alpha}_{\tilde{g}} = \rho^{-\frac{n-1}{2}}\abs{\alpha}_g = f^{-1}\abs{\alpha}_g = 1.
}
Therefore, with respect to $\tilde{g}$, we have
\eq{
    \tilde{f}=1, \quad \tilde{\beta}=\alpha, \quad \tilde{P}=0, \quad \tilde{Q}_X=X\ip\tilde{*}\alpha.
}
Consequently,
\eq{\label{eq:29.4}
    \tilde{S}_X = \tilde{\nabla}_X\alpha - X\ip\tilde{*}\alpha.
}
Using the Koszul formula for $\frac{n-1}{2}$-forms, we have
\eq{\label{eq:29.5}
    \tilde{\nabla}_X\alpha = \nabla_X\alpha - \frac{n-1}{2}X(\log\rho)\alpha - \rd(\log\rho)\w X\ip\alpha + X^\flat\w\nabla(\log\rho)\ip\alpha.
}
Since $\alpha=f\beta$, we have
\eq{\label{eq:30}
    f\nabla_X\beta = \nabla_X\alpha - X(f)\beta = \nabla_X\alpha - \frac{n-1}{2}X(\log\rho)\alpha.
}
Moreover, by definition of $P$ and $Q$ we have
\eq{\label{eq:31}
    fP_X = \rd(\log\rho)\w X\ip\alpha - X^\flat\w\nabla(\log\rho)\ip\alpha,
}
and by \eqref{conformal_change}
\eq{\label{eq:32}
    fQ_X = \rho X\ip*\alpha = X\ip\tilde{*}\alpha.
}
Combining \eqref{eq:29.4}, \eqref{eq:29.5}, \eqref{eq:30}, \eqref{eq:31}, and \eqref{eq:32} yields
\eq{
    fS_X &= f\nabla_X\beta - fP_X - fQ_X \\
    &= \nabla_X\alpha - \frac{n-1}{2}X(\log\rho)\alpha - \rd(\log\rho)\w X\ip\alpha + X^\flat\w\nabla(\log\rho)\ip\alpha - X\ip\tilde{*}\alpha \\
    &= \tilde{\nabla}_X\alpha - X\ip\tilde{*}\alpha = \tilde{S}_X,
}
the claim.
\end{proof}

Now continue to consider the equality case.
Equality in \eqref{eq:29} holds if and only if equality holds in both Theorem~\ref{thm:GN-Sn} and Lemma~\ref{thm4.7}. Hence, up to a conformal transformation, we may assume
\eq{
    \abs{\alpha}=1, \qquad S=0.
}
Here we used the conformal covariance \eqref{eq:28}. With respect to this conformal metric we then have
\eq{
    0 = S_X = \nabla_X\alpha - X\ip*\alpha.
}
Taking  wedge and trace and using Proposition~\ref{basic_form}(5) yields
\eq{
    \rd\alpha = e^i\w\nabla_{e_i}\alpha = e^i\w e_i\ip*\alpha = \frac{n+1}{2}*\alpha.
}
Thus $\alpha$ is a positive $\frac{n-1}{2}$-Killing form. The positivity requires the conformal transformation to be orientation-preserving. The converse is immediate.
\end{proof}

\section{Application 1. Optimal metrics on \texorpdfstring{$\S^n$}{Sn}}\label{sec0}

The first application is Theorem~\ref{thm:optimal-metrics-intro}, which is in fact equivalent to our main result.

Let $(M^n,g)$ be a closed oriented Riemannian manifold, and let
\eq{[g]\coloneqq\{\tilde g\mid \tilde g=u^2 g\}}
be its conformal class. Denote by $\lambda_1^+(\tilde g)$ the first positive eigenvalue of the curl operator with respect to $\tilde g\in[g]$, and define the conformal invariant
\eq{
    \mu([g]) \coloneqq \inf_{\tilde g\in[g]} \lambda_1^+(\tilde g)\,\rVol(\tilde g)^{\frac{1}{n}}.
}
It is known that $\mu([g])\in(0,\infty)$; see \cite{Jammes07}. Determining whether the infimum is achieved (and describing optimizers) appears to be difficult; this has remained open for a long time, even for $n=3$ and $(M,g)=(\S^3,g_{\rm st})$. As a consequence of Theorem~\ref{main_thm}, we show
\eq{
    \mu([g_{\mathbb{S}^n}])= \frac{n+1}{2}\,\omega_n^{\frac{1}{n}}.
}

The invariant $\mu([g])$ is closely related to the functional $J(\alpha)$.

\begin{lemma}\label{lem6.1}
\eq{
    \mu([g]) = \inf_{\int \< \curl\alpha, \alpha\> >0} J(\alpha,g).
}
\end{lemma}

\begin{proof}
The argument is standard, see \cite{EGP25} for $n=3$.  We include it for completeness.

Fix $\tilde g\in[g]$, and let $\lambda_1^+(\tilde g)$ and $\alpha_1$ be the first positive eigenvalue and a corresponding eigenform of the curl operator. Using the conformal invariance of $J$ and H\"older's inequality, we have
\eq{
    \inf J(\alpha,g)
    = \inf J(\alpha,\tilde g)
    \le J(\alpha_1,\tilde g)
    = \lambda_1^+(\tilde g)\,
    \frac{\bigl(\int \abs{\alpha_1}_{\tilde g}^{\frac{2n}{n+1}}\,\rdV_{\tilde g}\bigr)^{\frac{n+1}{n}}}{\int \abs{\alpha_1}_{\tilde g}^2\,\rdV_{\tilde g}}
    \le \lambda_1^+(\tilde g)\,\rVol(\tilde g)^{\frac{1}{n}}.
}
Taking the infimum over $\tilde g\in[g]$ gives $\inf J(\alpha,g)\le \mu([g])$.

For the reverse inequality, fix $\varepsilon>0$ and choose $\alpha_\varepsilon$ with $\int\<\curl\alpha_\varepsilon,\alpha_\varepsilon\> >0$ such that
\eq{
    J(\alpha_\varepsilon,g) \le \inf J(\alpha,g) + \varepsilon.
}
After scaling, we may assume
\eq{
    \int \abs{\curl_g\alpha_\varepsilon}_g^{\frac{2n}{n+1}}\rdV_g = 1.
}
Define $u\coloneqq \abs{\curl_g\alpha_\varepsilon}_g^{\frac{2}{n+1}}$ (assuming for the moment that $\curl_g\alpha_\varepsilon$ is nowhere vanishing) and set $\tilde g\coloneqq u^2 g$. Then
\eq{
    \rVol(\tilde g)=\int u^n\,\rdV_g = \int \abs{\curl_g\alpha_\varepsilon}_g^{\frac{2n}{n+1}}\,\rdV_g = 1.
}
Using the conformal covariance \eqref{conformal_change} we have
\eq{
    \Big(\int \abs{\curl_{\tilde g}\alpha_\varepsilon}_{\tilde g}^{\frac{2n}{n+1}}\rdV_{\tilde g}\Big)^{\frac{n+1}{n}}
    = 1 = \int \abs{\curl_{\tilde g}\alpha_\varepsilon}_{\tilde g}^2\,\rdV_{\tilde g}.
}
Therefore,
\eq{
    \inf J(\alpha,g) + \varepsilon
    \ge J(\alpha_\varepsilon,g)
    = J(\alpha_\varepsilon,\tilde g)
    = \frac{\int \abs{\curl_{\tilde g}\alpha_\varepsilon}_{\tilde g}^2\,\rdV_{\tilde g}}{\int \<\curl_{\tilde g}\alpha_\varepsilon,\alpha_\varepsilon\>_{\tilde g}\,\rdV_{\tilde g}}
    \ge \lambda_1^+(\tilde g)
    = \lambda_1^+(\tilde g)\,\rVol(\tilde g)^{\frac{1}{n}}
    \ge \mu([g]).
}
Letting $\varepsilon\to 0$ yields $\inf J(\alpha,g)\ge \mu([g])$.

If $u$ has zeros, one can instead take $u_\varepsilon\coloneqq \sqrt{u^2+\varepsilon^2}$ and run the same argument.
\end{proof}

A conformal metric $\tilde g\in[g]$ achieving $\mu([g])$ is called an \emph{optimizer}. There are interesting results on closed $3$-manifolds \cite{EGP25} (which proves local optimality in the $C^0$ topology) and on $3$-manifolds with boundary (where one also needs boundary conditions); see \cites{CDGT00,EGP22}. However, for $(M,g)=(\S^3,g_{\rm st})$ the existence and classification of optimizers has been open for a long time. As a consequence of Theorem~\ref{main_thm}, we now give an affirmative answer on the sphere and prove Theorem~\ref{thm:optimal-metrics-intro}.

\begin{proof}[Proof of Theorem~\ref{thm:optimal-metrics-intro}]
This is equivalent to our main theorem. Indeed, by Lemma~\ref{lem6.1} and Theorem~\ref{main_thm},
\eq{
    \mu([g_{{\rm st}}])= \frac{n+1}{2}\,\omega_n^{\frac{1}{n}}.
}
For the standard round metric one has
\[
\lambda_1^+(g_{{\rm st}})=\frac{n+1}{2}.
\]
Therefore $g_{{\rm st}}$ achieves the infimum $\mu([g_{{\rm st}}])$, hence is optimal. Uniqueness follows from the equality case of Theorem~\ref{main_thm}.
\end{proof}

Theorem~\ref{thm:optimal-metrics-intro} is a Hersch-type theorem \cite{Hersch70}, or much closer, an El~Soufi--Ilias type theorem \cite{ElSoufiIlias86}. Namely, Theorem~\ref{thm:optimal-metrics-intro} is equivalent to:

\smallskip

{\it Any metric $\tilde{g}\in[g_{{\rm st}}]$ with the given volume $\omega_n$ has the first positive eigenvalue \eq{ \lambda^+_1(\tilde{g})
\ge \frac{n+1} 2,}
with equality if and only if $\tilde{g}=g_{{\rm st}}$.}

\smallskip

For recent developments on the eigenvalues of  the $\curl$ operator, see \cites{B19,Montiel23IsoCurl,EGP25,BCV26WeylCurl} and the references therein.

\begin{remark}
    Similar to the Dirac operator $\D$ for spinor fields, the curl operator has also infinitely many negative eigenvalues.
    Let $\lambda_1^-(g)$ be the largest negative eigenvalue of $\curl$. Unlike the spinor case, $\abs{\lambda_1^-(g)}=\lambda_1^+(g)$ may not true, since the spectrum of $\curl$ may not be symmetric about $0$.
    Hence one can also consider the minimization problem
    \[
  \mu^-([g]) \coloneqq \inf_{\tilde g\in[g]} \abs{\lambda_1^-(\tilde g)}\,\rVol(\tilde g)^{\frac{1}{n}}.
    \]
    The same argument implies that
    \[
    \mu^-([g_{{\rm st}}])=\frac {n+1} 2 \omega_n^{\frac 1n}.
    \]
    Moreover, $\mu^-([g_{{\rm st}}])$ is achieved by eigenforms $\zeta$ with $\curl \zeta =-\frac {n+1}2 \zeta$ and their conformal changes.
\label{negative}
\end{remark}

Theorem \ref{main_thm} also implies the following result about the first eigenvalue of the Hodge Laplacian on $\frac{n-1} 2$-forms. Let $\lambda_1^{\delta}(g)$ be the first eigenvalue of the Hodge Laplacian
\eq{\rd\rd^*+\rd^*\rd}
in the space
\eq{\{\alpha \,|\, \rd^*\alpha=0\}.}
Due to the commutativity of $\rd^*$ and the Hodge Laplacian, $\lambda_1^{\delta}(g)$ is also an eigenvalue of the Hodge Laplacian. We have

\smallskip

\begin{corollary}
The standard round metric is the unique minimizer of $\lambda_1^{\delta}$ in its conformal class of metrics with the given volume $\omega_n$.
\end{corollary}

For this minimization  problem, see also \cites{ColboisElSoufi06, Jammes07, EGP25}.

\section{Application 2. The Hopf map minimizes the \texorpdfstring{$3$}{3}-energy}\label{app:3-energy}

Let $u:\S^3\to\S^2$ be a smooth map, and let $\omega_{\S^2}$ denote the volume form on $\S^2$. Since
\eq{
    \rd(u^*\omega_{\S^2}) = u^*(\rd\omega_{\S^2}) = 0,
}
the $2$-form $u^*\omega_{\S^2}$ is closed. As $H^2(\S^3)=0$, it is exact, so we may choose a $1$-form $\alpha\in\Omega^1(\S^3)$ such that
\eq{
    u^*\omega_{\S^2} = \rd\alpha.
}
If desired, $\alpha$ can be chosen co-closed by Hodge theory.

The Hopf invariant of $u$ is defined by
\eq{\label{eq:49}
    \int_{\S^3} \alpha\w\rd\alpha.
}
It is well-defined, i.e. independent of the choice of potential: if $\alpha' = \alpha+\rd f$, then
\eq{
    \alpha'\w\rd\alpha' = \alpha\w\rd\alpha + \rd f\w\rd\alpha,
}
and since $\rd f\w\rd\alpha=\rd(f\,\rd\alpha)$ is exact, its integral over the closed manifold $\S^3$ vanishes. It is also well known that \eqref{eq:49} is a homotopy invariant.

\medskip

\noindent\textit{Explicit formula for the Hopf map.}
Identify $\S^3\subset\mathbb{C}^2$ as
\eq{\label{eq:S3-in-C2}
    \S^3=\{(z_1,z_2)\in\mathbb{C}^2:|z_1|^2+|z_2|^2=1\}.
}
Then the (standard) Hopf map $\pi:\S^3\to\S^2\subset\R^3$ is given by
\eq{\label{eq:Hopf-map-explicit}
    \pi(z_1,z_2)=\bigl(2\,\mathrm{Re}(z_1\overline{z_2}),\;2\,\mathrm{Im}(z_1\overline{z_2}),\;|z_1|^2-|z_2|^2\bigr),
}
which indeed satisfies $|\pi(z_1,z_2)|=1$.

\medskip

The Hopf map $\pi:\S^3\to\S^2$ enjoys many special properties. In particular, since $\abs{\rd\pi}$ is constant, it is a critical point of the $p$-energy
\eq{\int_{\S^3} \abs{\rd u}^p}
for every $p\in(1,\infty)$ (equivalently, $\pi$ is $p$-harmonic for all $p$).

Rivi\`ere \cite{Riviere98CAG} proved that
$\pi$ enjoys a strong variational rigidity: for $p\ge 4$ it is the unique minimizer of the $p$-energy in its homotopy class, up to orientation-preserving isometries of $\S^3$. Moreover, there exists a $C^1$-neighborhood of $\pi$ in which $\pi$ is the unique minimizer of the $p$-energy (modulo orientation-preserving isometries if $p>3$, and modulo orientation-preserving conformal maps if $p=3$). For $p<3$, $\pi$ is not a local minimizer (e.g. by composing with dilations).

As an application of Theorem~\ref{main_thm}, we are able to prove Theorem~\ref{thm:hopf-3energy-intro}, which was conjectured by Rivi\`ere \cites{Riviere98CAG,Riviere23}. 

\begin{proof}[Proof of Theorem~\ref{thm:hopf-3energy-intro}]
Actually Rivi\`ere in \cite{Riviere98CAG} already proved that Theorem \ref{main_thm} implies Theorem \ref{thm:hopf-3energy-intro}. For convenience of the reader we provide his reduction and prove the rigidity result.

Since $H^2(\S^3)=0$, the closed pull-back $2$-form $u^*\omega_{\S^2}$ is exact, so
\eq{
    u^*\omega_{\S^2} = \rd\alpha
}
for some $1$-form $\alpha$ on $\S^3$. If desired, $\alpha$ can be chosen co-closed by Hodge theory.

The AM--GM inequality implies
\eq{\label{eq:23}
    \abs{\rd u}^2 \geq 2\abs{\rd\alpha},
}
with equality if and only if $u$ is horizontally weakly conformal. More precisely, choose pointwise orthonormal bases $\{e_1,e_2,e_3\}$ of $T_x\S^3$ and $\{f_1,f_2\}$ of $T_{u(x)}\S^2$ such that
\eq{
    \rd u(e_1)=\lambda_1f_1, \quad \rd u(e_2)=\lambda_2f_2, \quad \rd u(e_3)=0.
}
Then
\eq{
    \abs{\rd u}^2 = \lambda_1^2 + \lambda_2^2,
}
and
\eq{
    u^*\omega_{\S^2}(e_1,e_2) = \omega_{\S^2}(\rd u(e_1),\rd u(e_2)) = \omega_{\S^2}(\lambda_1f_1,\lambda_2f_2) = \lambda_1\lambda_2,
}
hence
\eq{
    u^*\omega_{\S^2} = \lambda_1\lambda_2 e^1\w e^2.
}
Therefore
\eq{
    \abs{\rd\alpha} = \abs{u^*\omega_{\S^2}} = \lambda_1\lambda_2 \leq \frac{1}{2}(\lambda_1^2 + \lambda_2^2) = \frac{1}{2}\abs{\rd u}^2,
}
with equality if and only if $\lambda_1=\lambda_2$. The estimate \eqref{eq:23} follows.

Recall that fixing the homotopy class is equivalent to fixing the Hopf invariant
\eq{
    \int_{\S^3} \alpha\w\rd\alpha = \int_{\S^3}\<\curl\alpha,\alpha\>.
}
Theorem~\ref{main_thm} and \eqref{eq:23} imply
\eq{
    \int_{\S^3} \abs{\rd u}^3
    \geq 2\sqrt{2}\int_{\S^3}\abs{\rd\alpha}^{\frac{3}{2}}
    \geq 2\sqrt{2}\cdot(2\omega_3^{\frac{1}{3}})^{\frac{3}{4}}\Big(\int_{\S^3}\< \curl\alpha,\alpha\>\Big)^{\frac{3}{4}}=\int_{\S^ 3} \abs{\rd \pi}^3,
}
because for the Hopf map $\pi$, all inequalities in the proof are equalities.

Moreover, equality holds if and only if $u$ is horizontally weakly conformal and $\alpha$ is a positive Killing $1$-form, modulo orientation-preserving conformal transformations and modulo $\ker(\rd)$; that is,
\eq{
    \alpha = \Phi^*\xi+\gamma,
}
where $\Phi\in{\rm Conf}^+(\S^3)$ is an orientation-preserving conformal transformation of $\S^3$, $\xi$ is a positive Killing $1$-form, and $\gamma\in\ker(\rd)$. The map $\Phi$ must be orientation-preserving since the Hopf invariant is fixed. Furthermore, fixing the Hopf invariant implies  fixing the length of $\xi$, and thus
\eq{
    \xi = R^*\xi_0 \quad\hbox{for some } R\in{\rm SO}(4),
}
where $\xi_0$ is the positive Killing $1$-form corresponding to the Hopf map $\pi$, i.e., $\rd\xi_0=\pi^*\omega_{\S^2}$. Hence
\eq{
    u^*\omega_{\S^2} = \rd\alpha = \Phi^*\rd\xi = \Phi^*R^*\rd\xi_0 = (\pi\circ R\circ \Phi)^*\omega_{\S^2}.
}
Since $\ker(\rd u)=\ker(\rd(\pi\circ R\circ\Phi))$, it follows that
\eq{
    u = h\circ\pi\circ R\circ \Phi,
}
for some $h:\S^2\to\S^2$ that preserves the area measure. Since both $u$ and $\pi\circ R\circ\Phi$ are horizontally weakly conformal, it follows that $h$ is conformal. As $h$ preserves the area measure, it must in fact be an isometry.
Finally, since the Hopf map $\pi$ is $\S^1$-invariant, $h$ admits an isometric lift to $\S^3$, i.e.
\eq{
    h\circ\pi = \pi\circ R', \qquad \hbox{for some } R'\in{\rm SO}(4).
}
Therefore,
\eq{
    u = h\circ\pi\circ R\circ\Phi = \pi\circ (R'\circ R\circ\Phi), \qquad R'\circ R\circ\Phi\in{\rm Conf}^+(\S^3).
}
The converse is immediate, which completes the proof.
\end{proof}

\section{Application 3. The Hopf map minimizes the Faddeev--Skyrme energy}\label{sec:app3}

The Faddeev--Skyrme energy on $\S^3$ is defined by
\eq{
    \mathcal{FS}_\rho(u) &\coloneqq \int_{\S^3} \abs{\rd u}^2\,\rdV + \frac{1}{4\rho^2}\int_{\S^3} \abs{\rd u\w\rd u}^2\,\rdV, \\
    &= \int_{\S^3} \abs{\rd u}^2\,\rdV + \frac{1}{\rho^2}\int_{\S^3} \abs{u^*\omega_{\S^2}}^2\,\rdV, \qquad u:\S^3\to\S^2, \ \rho>0.
}
We normalize the Hopf invariant by
\eq{\label{eq:48}
    \int_{\S^3} \alpha\w\rd\alpha = 8\omega_3,
}
where $\alpha\in\Omega^1(\S^3)$ is chosen co-closed and satisfies $u^*\omega_{\S^2}=\rd\alpha$. 

The goal is to minimize the Faddeev--Skyrme energy within the homotopy class. The existence of a minimizer was proved by Lin--Yang \cite{LinYang04}, see also Esteban \cite{Esteban86}.

Ward \cite{Ward99} showed that the Hopf map $\pi$ is unstable for $\rho>\sqrt{2}$. On the other hand, Speight--Svensson \cite{SpeightSvensson07} proved that $\pi$ is (linearly) stable when $\rho\le\sqrt{2}$, and Isobe \cite{Isobe08} proved that $\pi$ is a local minimizer when $\rho<\sqrt{2}$. It is therefore natural to conjecture that
\[
\textit{the Hopf map $\pi:\S^3\to\S^2$ is a global minimizer when $\rho\le\sqrt{2}$.}
\]
The case $\rho=0$  was proved in \cite{SpeightSvensson11}. Very recently, Guerra--Lamy--Zemas \cite{GLZ26} made a breakthrough by proving global minimality (and uniqueness modulo rigid motions) of the Hopf map for $\rho\le 1$. In fact, they improved it to prove that there exists $\rho_0\in (1, \sqrt 2]$ such that  global minimality holds for any  $\rho \le \rho_0$.

In this section, we prove the conjecture for all $\rho\le\sqrt{2}$, and we show that the only minimizers are of the form $\pi\circ R$ with $R\in{\rm SO}(4)$.

Let $\alpha_0$ be the $1$-form associated with the Hopf map $\pi$, i.e. $\pi^*\omega_{\S^2}=\rd\alpha_0$. Under the normalization $\int_{\S^3}\alpha_0\w\rd\alpha_0=8\omega_3$, we have
\eq{
    \abs{\rd\pi}^2 = 8, \qquad \abs{\pi^*\omega_{\S^2}} = \abs{\rd\alpha_0} = 4.
}
Therefore,
\eq{
    \mathcal{FS}_\rho(\pi) = 8\omega_3 + 16\omega_3\rho^{-2}.
}
Define
\eq{
    E_\rho(\alpha) \coloneqq 2\int_{\S^3}\abs{\rd\alpha}\,\rdV + \rho^{-2}\int_{\S^3}\abs{\rd\alpha}^2\,\rdV.
}
Then
\eq{\label{eq:24}
    E_\rho(\alpha_0) = 8\omega_3 + 16\omega_3\rho^{-2} = \mathcal{FS}_\rho(\pi),
}

\begin{lemma}\label{lem7.1}
Let $u:\S^3\to\S^2$ and let $\alpha\in\Omega^1(\S^3)$ satisfy $u^*\omega_{\S^2}=\rd\alpha$. Then
\eq{
    \mathcal{FS}_\rho(u) \geq E_\rho(\alpha),
}
with equality if and only if $u$ is horizontally weakly conformal. In particular, the Hopf map $\pi$ attains equality.
\end{lemma}

\begin{proof}
By \eqref{eq:23} we have the pointwise estimate $\abs{\rd u}^2\ge 2\abs{\rd\alpha}$. Integrating yields
\eq{
    \int_{\S^3}\abs{\rd u}^2\,\rdV \ge 2\int_{\S^3}\abs{\rd\alpha}\,\rdV.
}
Since $\mathcal{FS}_\rho(u)=\int_{\S^3}\abs{\rd u}^2\,\rdV + \rho^{-2}\int_{\S^3}\abs{\rd\alpha}^2\,\rdV$, this implies $\mathcal{FS}_\rho(u)\ge E_\rho(\alpha)$.

Moreover, equality holds if and only if equality holds in \eqref{eq:23} almost everywhere, i.e. if and only if $u$ is horizontally weakly conformal.
\end{proof}

We now restate Theorem \ref{thm:fs-intro}.
\begin{theorem}\label{thm:FS-minimizers}
For every $\rho\le\sqrt{2}$, the minimizers of $\mathcal{FS}_{\rho}$ in the homotopy class of the Hopf map $\pi$ are precisely $\pi\circ R$ with $R\in{\rm SO}(4)$.
\end{theorem}
\begin{proof}
By Lemma~\ref{lem7.1} and \eqref{eq:24}, it suffices to prove minimality and uniqueness for $E_\rho$ when $\rho\le\sqrt{2}$. Moreover,
it suffices to prove that $\alpha_0$ minimizes $E_{\sqrt{2}}$ and to establish uniqueness. Indeed, H\"older's inequality, the normalization \eqref{eq:48}, and Theorem~\ref{main_thm} imply
\eq{
    \int_{\S^3} \abs{\rd\alpha}^2 \ge \omega_3^{-\frac{1}{3}}\Big( \int_{\S^3}\abs{\rd\alpha}^{\frac{3}{2}} \Big)^{\frac{4}{3}} \ge 16\omega_3.
}
Assuming the case $\rho=\sqrt{2}$, we then obtain for $\rho\le\sqrt{2}$,
\eq{\label{eq:25}
    E_\rho(\alpha) &= E_{\sqrt{2}}(\alpha) + \Big(\rho^{-2}-\frac{1}{2}\Big)\int_{\S^3} \abs{\rd\alpha}^2 \\
    &\ge 16\omega_3 + 16\omega_3\Big(\rho^{-2}-\frac{1}{2}\Big) \\
    &= 8\omega_3 + 16\omega_3\rho^{-2} = E_\rho(\alpha_0).
}
Moreover, equality for some $\rho\le\sqrt{2}$ implies equality for $\rho=\sqrt{2}$.

We now consider the case $\rho=\sqrt{2}$.
Using Theorem~\ref{main_thm}, we have
\eq{
    \Big(\int_{\S^3} \abs{\rd\alpha}^{\frac{3}{2}}\Big)^{\frac{4}{3}} \ge 2\omega_3^{\frac{1}{3}}\int_{\S^3}\alpha\w\rd\alpha = 16\omega_3^{\frac{4}{3}}.
}
Therefore,
\eq{\label{eq:eq1}
    E_{\sqrt{2}}(\alpha) = 2\int_{\S^3}\abs{\rd\alpha} + \frac{1}{2}\int_{\S^3}\abs{\rd\alpha}^2 \ge 2\int_{\S^3} \abs{\rd\alpha}^{\frac{3}{2}} \ge 16\omega_3 = E_{\sqrt{2}}(\alpha_0).
}

We next characterize minimizers of $\mathcal{FS}_{\sqrt{2}}$. Each minimizer $u$ is also a minimizer of $E_{\sqrt{2}}$, and satisfies
\eq{
    \mathcal{FS}_{\sqrt{2}}(u)=E_{\sqrt{2}}(\alpha)=16\omega_3.
}
By Lemma~\ref{lem7.1}, $u$ is horizontally weakly conformal. Moreover, by the equality case of Theorem~\ref{main_thm} and AM--GM, equality in \eqref{eq:eq1} holds if and only if
\eq{
    \abs{\rd\alpha}\equiv4, \quad \alpha = \Phi^*\xi+\gamma,
}
where $\Phi$ is an orientation-preserving conformal transformation of $\S^3$, $\xi$ is a positive Killing $1$-form, and $\gamma\in\ker(\rd)$. Using the same argument as in the proof of equality case of Theorem~\ref{thm:hopf-3energy-intro}, we see that $u=\pi\circ R$ for some $R\in{\rm SO}(4)$. See also \cite[Proposition 5.1]{GLZ26}. The converse is immediate, which completes the proof.
\end{proof}

\begin{remark}As mentioned above, that Theorem\ref{main_thm} implies Theorem \ref{thm:fs-intro} is known to experts. See for example \cite{Harland13MassivePions}, where the author also conjectured the best constant in Theorem \ref{main_thm} when $n=3$.
\end{remark}

\section{Application 4. A sharp lower bound for the magnetic field of zero modes}\label{sec:app4}

For convenience, we restate the zero mode equation \eqref{eq1.0}. In this section, we write
\eq{\label{zero-mode}
    \D\varphi = i\,A\cdot\varphi.
}
A solution $\varphi\not\equiv 0$ of \eqref{zero-mode} is called a \emph{zero mode}. For background and physical motivation, we refer to \cite{Loss_Yau_86,FL1,FL22,Frank_Loss_2024} and the references therein.

Equation \eqref{zero-mode} enjoys two fundamental invariances.

\smallskip
\noindent\emph{Gauge invariance.} For any real-valued function $f$, the transformation
\eq{
    (\varphi,A)\longmapsto (e^{if}\varphi,\,A+\nabla f)
}
sends solutions of \eqref{zero-mode} to solutions.

\smallskip
\noindent\emph{Conformal invariance.} If $\tilde g=\rho^2g$, then
\eq{
    (\varphi,A,g)\longmapsto (\rho^{-\frac{n-1}{2}}\tilde\varphi,\,\rho^{-2}A,\,\tilde g)
}
also preserves the solution set, where $\varphi\mapsto\tilde\varphi$ denotes the natural isometry between the spinor bundles associated with $g$ and $\tilde g$ (see, e.g., \cite{WZ25b}).

In \cite{Frank_Loss_2024}, Frank--Loss proved a sharp necessary condition for the existence of zero modes in $\R^3$ (and, by conformal equivalence, on $\S^3$), namely
\eq{\label{A}
    \norm{A}_3 \ge \frac{3}{2}\,\omega_3^{\frac{1}{3}}.
}
More generally, analogous criteria are available in higher dimensions and on manifolds; see \cite{Frank_Loss_2024,Reuss25,WZ25b}.
In this section, we restrict ourselves to the  $3$-dimensional case.

In addition, Frank--Loss \cites{FL1,Frank_Loss_2024} established the universal lower bound
\eq{\label{eq:54}
    \norm{\curl A}_{\frac{3}{2}} \geq \frac{3}{2}\,\omega_3^{\frac{2}{3}}.
}
They noted in \cite{FL22} that the constant in \eqref{eq:54} is not optimal, and this was subsequently confirmed by Julio--Batalla \cite{Julio26}.
Here $\curl A$ is the \emph{magnetic field}; geometrically, it is the curvature of the connection $\rd+iA$ on the trivial complex line bundle over $\S^3$. Hence it is gauge invariant, so it is more interesting to seek a best lower bound for $\curl A$. A distinguished family of solutions of \eqref{zero-mode} is provided by Killing spinors together with their associated Reeb fields; for these examples one has
\[
    \norm{\curl A}_{\frac{3}{2}} = 3\,\omega_3^{\frac{2}{3}}
\]
(see Subsection~\ref{sec8.1}). It is therefore natural to conjecture that $3\,\omega_3^{\frac{2}{3}}$ is the minimal possible value.

Theorem~\ref{thm:magnetic} confirms this conjecture.

\subsection{Preliminaries on zero modes} \label{sec8.1}

We briefly review the notation for the Dirac operator and the zero mode equation. For details, see \cite{WZ25b}.
Let $\Sigma$ be the spinor bundle over $\S^3$.
We write $\<\cdot,\cdot\>$ for the Hermitian inner product on the spinor bundle, and $\rRe\<\cdot,\cdot\>$ for its real part. The Clifford action of a vector field $X$ on a spinor field $\varphi$ is denoted by $X\cdot\varphi$. The dual $1$-form gives the same action, i.e. $X^\flat\cdot\varphi=X\cdot\varphi$.

Fix the orientation so that, for a local oriented orthonormal frame $\{e_1,e_2,e_3\}$,
\eq{\label{eq:59}
    e_1\cdot e_2\cdot = e_3\cdot, \quad e_2\cdot e_3\cdot = e_1\cdot, \quad e_3\cdot e_1\cdot = e_2\cdot.
}
In particular, for any $1$-form $\gamma$ we have
\eq{
    \gamma\cdot = (*\gamma)\cdot.
}
Let $\nabla$ be the induced spin connection and $\D$ the Dirac operator, i.e.,
\[
\D= e_i\cdot \nabla_{e_i}.
\]
We define the modified spin connection by
\eq{
    \nabla^A_X\varphi \coloneqq \nabla_X\varphi - i\<A,X\>\varphi.
}
It is a metric Clifford connection, i.e. compatible with the metric and the Clifford multiplication. The corresponding Dirac operator is
\eq{
    \D^A\varphi \coloneqq e_j\cdot\nabla^A_{e_j}\varphi = \D\varphi - iA\cdot\varphi.
}
Thus the zero mode equation becomes
\eq{\label{eq:67}
    \D^A\varphi = 0.
}
The Schr\"odinger--Lichnerowicz formula for $\D^A$  states
\eq{\label{eq:84a}
    (\D^A)^2\varphi = (\nabla^A)^*\nabla^A\varphi + \frac{3}{2}\varphi - i\,\rd A^\flat\cdot\varphi,
}
see for instance \cite{Lawson}.

Now we give the important example of the zero mode and compute it carefully.
Let $\varphi_0$ be a unit $-\frac{1}{2}$-Killing spinor on $(\S^3,g_{\rm st})$, i.e.
\eq{
    \nabla_X\varphi_0 = -\frac{1}{2}X\cdot\varphi_0,\qquad \forall X\in\Gamma(T\S^3),\qquad \abs{\varphi_0}=1.
}
Define the associated $1$-form (identified with a vector field)
\eq{
    \xi_0 \coloneqq -\rRe\<ie_j\cdot\varphi_0,\varphi_0\>e_j.
}
Then $\xi_0$ is a unit Killing vector field (the Reeb field associated to $\varphi_0$) and satisfies
\eq{\label{eq:56}
    \curl \xi_0 = 2\xi_0, \qquad \xi_0\cdot\varphi_0 = i\varphi_0, \qquad \abs{\xi_0}=1.
}
Consequently,
\eq{
    \D\varphi_0 = \frac{3}{2}\varphi_0 = -\frac{3}{2}i\,\xi_0\cdot\varphi_0.
}
Therefore, set $A_0=-\frac{3}{2}\xi_0$, then $(\varphi_0,A_0)$ solves the zero mode equation. Using \eqref{eq:56} we obtain
\eq{
    \norm{\curl A_0}_{\frac{3}{2}} = 3\,\omega_3^{\frac{2}{3}}.
}
\begin{remark}\label{rmk:verify-84a}
For completeness, we verify \eqref{eq:84a} for $(\varphi_0,A_0)$. Since $\D^{A_0}\varphi_0=\D\varphi_0-iA_0\cdot\varphi_0=0$, the left-hand side of \eqref{eq:84a} vanishes. Moreover, $\curl\xi_0=2\xi_0$ implies $\rd\xi_0^\flat=2\,*\xi_0^\flat$, hence $\rd A_0^\flat=-\frac{3}{2}\rd\xi_0^\flat=-3\,*\xi_0^\flat$ and, by $\gamma\cdot=(*\gamma)\cdot$, we have
\eq{
    -i\,\rd A_0^\flat\cdot\varphi_0 = 3i\,\xi_0\cdot\varphi_0 = -3\varphi_0.
}
Thus \eqref{eq:84a} reduces to $(\nabla^{A_0})^*\nabla^{A_0}\varphi_0=\frac{3}{2}\varphi_0$, which is clearly true.
\end{remark}

\subsection{Proof of Theorem~\ref{thm:magnetic}}

Set
\eq{
    f\coloneqq\abs{\varphi}, \qquad \varphi = f\phi, \qquad \hbox{ with   }\abs{\phi}=1.
}
To illustrate the approach, we first assume that $f>0$ on $\S^3$.
Certainly, in general, $f$ may have zeros. Unlike in Section~\ref{sec:main}, we must treat the zero set with some care; see Subsection~\ref{sec:regularization}.

We start with a key $3$-dimensional linear-algebra fact. Since $\phi\neq 0$, define
\eq{
    \phi^\perp \coloneqq \{\psi\in\Gamma(\Sigma) \mid \rRe\<\psi,\phi\>=0\}.
}

\begin{lemma}\label{lem8.2}
The map $T\S^3\to\phi^\perp$, $X\mapsto X\cdot\phi$, is an isometry with respect to the Riemannian metric on $T\S^3$ and the real inner product $\rRe\<\cdot,\cdot\>$ on spinors. In particular,
$\{\phi,e_1\cdot\phi,e_2\cdot\phi,e_3\cdot\phi\}$ is an orthonormal basis of the spinor bundle over $\S^3$, with respect to the real inner product $\rRe\<\cdot,\cdot\>$.

\end{lemma}
\begin{proof}
If $X\cdot\phi=0$ then taking length yields $X=0$, so the map is injective. Since the spinor bundle over $\S^3$ has complex rank $2$, $\phi^\perp$ has real rank $4-1=3$, which equals ${\rm rank}_\R(T\S^3)$; therefore, the map is bijective. Finally, $\abs{X\cdot\phi}=\abs{X}\,\abs{\phi}=\abs{X}$ since $\abs{\phi}=1$, so it is an isometry.
\end{proof}

\begin{remark}\label{rmk8.3} The previous lemma indicates that when $n=3$, spinor fields can be identified with a vector field plus a scalar function. This correspondence becomes more complicated in higher dimensions. 
\end{remark}

Since $\abs{\phi}=1$, we have $\nabla^A_X\phi\perp\phi$ for all $X$. As a consequence of Lemma~\ref{lem8.2}, there exists an endomorphism $K\in\mathrm{End}(T\S^3)$ such that
\eq{\label{eq:64}
    \nabla^A_X\phi = K(X)\cdot\phi, \quad X\in\Gamma(T\S^3).
}
With respect to a local oriented orthonormal frame $\{e_1,e_2,e_3\}$, we write
\eq{\label{eq:def-Kij}
    K_{ij}\coloneqq \<K(e_i),e_j\>,
}
so that $K(e_i)=K_{ij}e_j$.
Hence
\eq{\label{eq:60}
    \nabla^A_X\varphi = \nabla^A_X(f\phi) = X(f)\phi + fK(X)\cdot\phi.
}

\begin{lemma}\label{lem8.3}
We have
\eq{
    e_j\cdot K(e_j)\cdot\phi = -\nabla(\log f)\cdot\phi,
}
moreover ${\rm tr}(K)=0$, and
\eq{\label{eq:65}
    -\nabla(\log f) = (K_{23}-K_{32},K_{31}-K_{13},K_{12}-K_{21}) = 2(H_{23},H_{31},H_{12}),
}
where $H=\frac{1}{2}(K-K^T)$ is the anti-symmetric part of $K$.
\end{lemma}
\begin{proof}
The zero mode equation implies
\eq{
    0 = \D^A\varphi = e_j\cdot\nabla^A_{e_j}\varphi = e_j\cdot(e_j(f)\phi+fK(e_j)\cdot\phi) = \nabla f\cdot\phi + fe_j\cdot K(e_j)\cdot\phi,
}
hence
\eq{\label{eq:57}
    e_j\cdot K(e_j)\cdot\phi = -\nabla(\log f)\cdot\phi.
}
Taking the real inner product with $\phi$ and using that Clifford multiplication by a vector is skew-Hermitian, we obtain
\eq{
    0 = \rRe\<e_j\cdot K(e_j)\cdot\phi,\phi\>.
}
On the other hand,
\eq{
    \rRe\<e_j\cdot K(e_j)\cdot\phi,\phi\>
    = \sum_{j} \rRe\<K(e_j)\cdot\phi, e_j\cdot\phi\>
    = -\sum_j K_{jj} = -{\rm tr}(K),
}
and thus ${\rm tr}(K)=0$. Finally, by \eqref{eq:59} we have
\eq{
    e_j\cdot K(e_j)\cdot\phi = \sum_{j,k}K_{jk}e_j\cdot e_k\cdot\phi = \sum_{j\not=k}K_{jk}e_j\cdot e_k\cdot\phi \\
    = (K_{12}-K_{21})e_3\cdot\phi + (K_{23}-K_{32})e_1\cdot\phi + (K_{31}-K_{13})e_2\cdot\phi.
}
The claim follows. 
\end{proof}

Decompose $K$ by
\eq{
    K = \frac{1}{3}{\rm tr}(K)\rid + H + S,
}
the trace part, anti-symmetric part, and symmetric trace-free part, with
\eq{
    H^T = -H, \qquad S^T=S, \qquad {\rm tr}(S)=0.
}
In view of the previous Lemma,
\eqref{eq:60} becomes
\eq{\label{eq:61}
    \nabla^A_X\varphi = X(f)\phi + fH(X)\cdot\phi + fS(X)\cdot\phi \eqcolon P_X+Q_X+R_X.
}
For any $P^1,P^2\in\Gamma(T^*\S^3\otimes\Sigma)$, define the real inner product by
\[
\rRe\langle P^1,P^2\rangle \coloneqq \rRe\langle P^1_{e_i}, P^2_{e_i}\rangle=\sum_{j=1}^3 \rRe\langle P^1_{e_j},P^2_{e_j}\rangle.
\]
\begin{lemma}\label{lem8.4}
 $P,Q,R$ are pairwise orthogonal with respect to the real inner product, i.e.
\eq{
    \rRe\<P,Q\>=\rRe\<P,R\>=\rRe\<Q,R\>=0.
}
\end{lemma}
\begin{proof}
First,
\eq{
    \rRe\<P,Q\> = e_j(f)f\rRe\<\phi,H(e_j)\cdot\phi\> = 0.
}
Second,
\eq{
    \rRe\<P,R\> = e_j(f)f\rRe\<\phi,S(e_j)\cdot\phi\> = 0.
}
Third, since $H$ is anti-symmetric and $S$ is symmetric,
\eq{
    \rRe\<Q,R\> = f^2\rRe\<H(e_j)\cdot\phi,S(e_j)\cdot\phi\> = f^2\<H,S\> = 0.
}
Hence the claim follows.
\end{proof}

\begin{lemma}\label{lem8.5}
We have
\eq{
    \abs{P}^2 = \abs{\nabla f}^2, \qquad \abs{Q}^2 = \frac{1}{2}\abs{\nabla f}^2, \qquad \abs{R}^2 = f^2\abs{S}^2.
}
\end{lemma}
\begin{proof}
Since $\abs{\phi}=1$, the first and last identities are trivial. Using \eqref{eq:65} we have
\eq{
    \abs{\nabla(\log f)}^2 = 4(H_{12}^2+H_{23}^2+H_{31}^2) = 2\abs{H}^2,
}
hence
\eq{
    \abs{Q}^2 = f^2\abs{H}^2 = \frac{1}{2}f^2\abs{\nabla(\log f)}^2 = \frac{1}{2}\abs{\nabla f}^2,
}
the claim.
\end{proof}

The following identity follows from \eqref{eq:61}, Lemma~\ref{lem8.4}, and Lemma~\ref{lem8.5}.
\begin{corollary}
We have
\eq{\label{eq:62}
\abs{\nabla^A\varphi}^2 = \frac{3}{2}\abs{\nabla f}^2 + f^2\abs{S}^2.
}
Hence dropping $f^2\abs{S}^2$ loses too much in the estimate.
\end{corollary}

\begin{remark} \label{rem8.9}
If we drop the remainder term $f^2\abs{S}^2$ in \eqref{eq:62}, then we recover the refined Kato inequality
\eq{
    \abs{\nabla^A\varphi}^2 \ge \frac{3}{2}\abs{\nabla f}^2,
}
which leads to the non-sharp estimate \eqref{eq:54} of Frank--Loss. However, we will see that the remainder $f^2\abs{S}^2$ contributes at the same scale, though not in a direct way. This is one of the key points in the proof. 

One checks the identity \eqref{eq:62} in the model case $(\varphi,A)=(\varphi_0,A_0)$. Here $f=|\varphi_0|\equiv 1$, hence $\nabla f=0$ and \eqref{eq:62} reduces to $|\nabla^{A_0}\varphi_0|^2=|S|^2$. Using $\nabla_X\varphi_0=-\frac{1}{2}X\cdot\varphi_0$, $A_0=-\frac{3}{2}\xi_0$, and $\xi_0\cdot\varphi_0=i\varphi_0$, we compute
\eq{
    \nabla^{A_0}_X\varphi_0=\nabla_X\varphi_0-i\<A_0,X\>\varphi_0=
    \Bigl(-\frac{1}{2}X+\frac{3}{2}\<\xi_0,X\>\xi_0\Bigr)\cdot\varphi_0.
}
In the oriented orthonormal frame $\{e_1,e_2,e_3=\xi_0\}$ this gives $K=\mathrm{diag}(-\tfrac12,-\tfrac12,1)$, hence $H=0$ and $S=K$. Therefore
\eq{
    |\nabla^{A_0}\varphi_0|^2=\sum_{j=1}^3|K(e_j)|^2=\frac14+\frac14+1=\frac32=|S|^2,
}
which agrees with \eqref{eq:62}.

\end{remark}

Next, define a vector field $\xi$ by
\eq{\label{eq:63}
    \xi \coloneqq \rRe\<i\phi,e_j\cdot\phi\>e_j.
}
By Lemma~\ref{lem8.2} we have
\eq{\label{eq:83}
    \xi\cdot\phi = \rRe\<i\phi,e_j\cdot\phi\>e_j\cdot\phi = i\phi,
}
so $\abs{\xi}=1$. Fix a point and choose an oriented orthonormal frame $\{e_1,e_2,e_3\coloneqq\xi\}$ with $\nabla e_j=0$ at that point. Denote by $S_{jk}$ the components of $S$ in this frame. Since $S$ is symmetric and ${\rm tr}(S)=0$,
\eq{
    S_{11}+S_{22}+S_{33}=0.
}
Moreover, by AM--GM,
$
    S_{11}^2+S_{22}^2 \ge \frac12(S_{11}+S_{22})^2 = \frac12 S_{33}^2.
$
Consequently,
\eq{\label{eq:66}
    \abs{S}^2 &= \sum_{j,k}S_{jk}^2 = (S_{11}^2+S_{22}^2+S_{33}^2) + 2(S_{13}^2+S_{23}^2) + 2S_{12}^2 \\
    &\geq \frac{3}{2}S_{33}^2 + \frac{3}{2}(S_{13}^2+S_{23}^2) + 2S_{12}^2 \geq \frac{3}{2}(S_{13}^2+S_{23}^2+S_{33}^2),
}
with equality if and only if
\eq{\label{eq:82}
    S_{11}=S_{22}=-\frac{1}{2}S_{33}, \qquad S_{12}=S_{13}=S_{23}=0.
}

\begin{lemma}\label{lem8.9}
Set $1$-form $\alpha\coloneqq f\xi^\flat$, then
\eq{
    \abs{\rd\alpha}^2 = 4f^2(S_{13}^2+S_{23}^2+S_{33}^2).
}
Consequently, we have 
\eq{\label{eq:68}
   f^2 \abs{S}^2 \geq \frac{3}{8}\abs{\rd\alpha}^2,
}
where $\abs{\rd\alpha}^2=\sum_{j<k}\rd\alpha(e_j,e_k)^2$.
\end{lemma}
\begin{proof}
Differentiating \eqref{eq:63} and using \eqref{eq:64} yield
\eq{
    (\nabla_{e_k}\xi^\flat)(e_1) &= \rRe\<i\nabla^A_{e_k}\phi,e_1\cdot\phi\> + \rRe\<i\phi,e_1\cdot\nabla^A_{e_k}\phi\> \\
    &= \rRe\<iK(e_k)\cdot\phi,e_1\cdot\phi\> + \rRe\<i\phi,e_1\cdot K(e_k)\cdot\phi\> \\
    &= \rRe\<i(K_{k2}e_2+K_{k3}e_3)\cdot\phi,e_1\cdot\phi\> + \rRe\<i\phi,e_1\cdot(K_{k2}e_2+K_{k3}e_3)\cdot\phi\> \\
    &= 2K_{k2}\rRe\<i\phi,e_1\cdot e_2\cdot\phi\> + 2K_{k3}\rRe\<i\phi,e_1\cdot e_3\cdot\phi\> \\
    &= 2K_{k2}\rRe\<i\phi,\xi\cdot\phi\> + 2K_{k3}\rRe\<i\phi,e_1\cdot \xi\cdot\phi\> \\
    &= 2K_{k2}\rRe\<i\phi,i\phi\> + 2K_{k3}\rRe\<i\phi,ie_1\cdot\phi\> 
    = 2K_{k2}.
}
 Similarly, we have
\eq{
    (\nabla_{e_k}\xi^\flat)(e_2) = -2K_{k1}, \qquad (\nabla_{e_k}\xi^\flat)(e_3) = 0.
}
Therefore, using
\eq{
    \rd\alpha = \rd(f\xi^\flat) = \rd f\w\xi^\flat + f\rd\xi^\flat,
}
we have
\eq{
    \rd\alpha(e_1,e_2) &= f\rd\xi^\flat(e_1,e_2) = f(-2K_{11}-2K_{22}) = 2fK_{33} = 2fS_{33}, \\
    \rd\alpha(e_2,e_3) &= e_2(f) + f\rd\xi^\flat(e_2,e_3) = f(K_{13}-K_{31}) + 2fK_{31} = f(K_{13}+K_{31}) = 2fS_{13}, \\
    \rd\alpha(e_3,e_1) &= -e_1(f) + f\rd\xi^\flat(e_3,e_1) = f(K_{23}-K_{32}) + 2fK_{32} = f(K_{23}+K_{32}) = 2fS_{23},
}
where we used    \eqref{eq:65} to compute $e_1(f)$ and $e_2(f)$. The first statement follows.
In view of \eqref{eq:66}, \eqref{eq:68} is a direct consequence of it.
\end{proof}

Now we prove Theorem~\ref{thm:magnetic}.

\begin{proof}[Proof of Theorem~\ref{thm:magnetic}]
On one hand, we use the Schr\"odinger--Lichnerowicz formula \eqref{eq:84a} for $\D^A$ 
and the zero mode equation \eqref{eq:67} to obtain 
\eq{\label{eq:69}
    0 = \rRe\<(\D^A)^2\varphi,\varphi\> = \rRe\<(\nabla^A)^*\nabla^A\varphi,\varphi\> + \frac{3}{2}f^2 - \rRe\<i\,\rd A^\flat\cdot\varphi,\varphi\>.
}
On the other hand, using
\eq{
    -\frac{1}{2}\Delta(f^2) = -\frac{1}{2}\Delta\abs{\varphi}^2 = \rRe\<(\nabla^A)^*\nabla^A\varphi,\varphi\> - \abs{\nabla^A\varphi}^2,
}
we have
\eq{\label{eq:70}
    \rRe\<(\nabla^A)^*\nabla^A\varphi,\varphi\> = -\frac{1}{2}\Delta(f^2) + \abs{\nabla^A\varphi}^2 = -f\Delta f - \abs{\nabla f}^2 + \abs{\nabla^A\varphi}^2.
}
Combining \eqref{eq:62}, \eqref{eq:69}, and \eqref{eq:70} gives
\eq{\label{eq:71}
    f^2\rRe\<i\,\rd A^\flat\cdot\phi,\phi\> = \rRe\<i\rd A^\flat\cdot\varphi,\varphi\> = -f\Delta f + \frac{1}{2}\abs{\nabla f}^2 + \frac{3}{2}f^2 + f^2\abs{S}^2.
}
Set $f=u^2$. Then \eqref{eq:71} becomes
\eq{\label{eq:72}
    u^2\rRe\<i\,\rd A^\flat\cdot\phi,\phi\> = -2u\Delta u + \frac{3}{2}u^2 + u^2\abs{S}^2
=2(-u\Delta u+\frac 3 4 u^2) +f|S|^2.}
Since $\rd A^\flat\cdot\phi=\ast\rd A^\flat\cdot\phi=\curl A\cdot\phi$, the Cauchy--Schwarz inequality gives
\eq{\label{eq:73}
    u^2 \rRe\<i\,\rd A^\flat\cdot\phi,\phi\> = f \rRe\<i\curl A\cdot\phi,\phi\> \leq f\abs{\curl A}.
}
The critical Sobolev inequality \eqref{eq:S} gives
\eq{\label{eq:74}
    \int_{\S^3} (-u\Delta u + \frac 34u^2) \geq \frac 34 \,\omega_3^{\frac{2}{3}} \Big( \int_{\S^3} u^6 \Big)^{\frac{1}{3}}.
}
Integrating \eqref{eq:72} and using \eqref{eq:73}, \eqref{eq:74}, we have
\eq{\label{eq:79}
    \int_{\S^3} f\abs{\curl A} \geq \frac{3}{2}\omega_3^{\frac{2}{3}} \norm{f}_3 + \int_{\S^3} f\abs{S}^2.  
}
Using H\"older's inequality we have
\eq{\label{eq:80}
    \int_{\S^3} f\abs{\curl A} \leq \norm{f}_3\cdot\norm{\curl A}_{\frac{3}{2}}.
}
\eqref{eq:79} and \eqref{eq:80} imply
\eq{\label{eq:a2}
\norm{\curl A}_{\frac{3}{2}} \ge \frac{3}{2} \omega_3^{\frac{2}{3}} + \norm{f}_3^{-1} \int_{\S^3} f\abs{S}^2.
}
As remarked above, if we drop the last term, then the above inequality gives the estimate obtained in \cite{FL1}. We now claim 

\eq{\label{eq:claim}
\norm{f} _3^{-1} \int_{\S^3} f\abs{S}^2 \ge \frac 32  \omega_3^{\frac 23}.
}

To show the claim,
 we apply Theorem~\ref{thm:optimal-metrics-intro} to $\tilde{g}=f^2g_{{\rm st}}$ and $\alpha$. It, together with Remark \ref{negative},implies
\eq{\label{eq:75}
    \lambda_1^+(\tilde{g})\rVol(\tilde{g})^{\frac{1}{3}} \geq 2\omega_3^{\frac{1}{3}},\qquad \abs{\lambda_1^-(\tilde{g})}\,\rVol(\tilde{g})^{\frac{1}{3}} \geq 2\omega_3^{\frac{1}{3}}.    
}
Using the conformal covariance \eqref{conformal_change} we have
\eq{\label{eq:76}
    \rVol(\tilde{g}) = \int_{\S^3} f^3, \qquad \int_{\S^3}\abs{\rd\alpha}^2_{\tilde{g}}\rdV_{\tilde{g}} = \int_{\S^3} f^{-1}\abs{\rd\alpha}^2, \qquad  \int_{\S^3}\abs{\alpha}^2_{\tilde{g}}\rdV_{\tilde{g}} = \int_{\S^3} f\abs{\alpha}^2 = \int_{\S^3} f^3.
}
From the Rayleigh quotient, we have
\eq{\label{eq:93}
    \frac{\int_{\S^3}\abs{\rd\alpha}^2_{\tilde{g}}\rdV_{\tilde{g}}}{\inf_\phi\int_{\S^3}\abs{\alpha+\rd\phi}^2_{\tilde{g}}\rdV_{\tilde{g}}}  \ge  \min\{\lambda_1^+(\tilde{g})^2,\lambda_1^-(\tilde{g})^2\},}
where by \eqref{eq:76}
\eq{
    \inf_\phi\int_{\S^3}\abs{\alpha+\rd\phi}^2_{\tilde{g}}\rdV_{\tilde{g}} = \inf_\phi\int_{\S^3}f\abs{\alpha+\rd\phi}^2 = \inf_\phi\int_{\S^3} \Big(f\abs{\alpha}^2 + 2\<f\alpha,\rd\phi\> + f\abs{\rd\phi}^2 \Big).
}
Note that using Proposition~\ref{basic_form} (4) and the zero mode equation we have
\eq{
    \rd^*(f\alpha) &= \rd^*(f^2\xi^\flat) = -e_j\ip\nabla_{e_j}(f^2\xi^\flat) = -e_j(\rRe\<ie_j\cdot\varphi,\varphi\>) \\
    &= -\rRe\<ie_j\cdot\nabla_{e_j}\varphi,\varphi\> - \rRe\<ie_j\cdot\varphi,\nabla_{e_j}\varphi\> \\
    &= -\rRe\<i\D\varphi,\varphi\> + \rRe\<i\varphi,\D\varphi\> \\
    &= \rRe\<A\cdot\varphi,\varphi\> + \rRe\<\varphi,A\cdot\varphi\> = 0.
}
Hence
\eq{
    \int_{\S^3} \<f\alpha,\rd\phi\> = \int_{\S^3} \<\rd^*(f\alpha),\phi\> = 0.
}
It follows
\eq{
    \inf_\phi\int_{\S^3}\abs{\alpha+\rd\phi}^2_{\tilde{g}}\rdV_{\tilde{g}} = \int_{\S^3} f\abs{\alpha}^2 = \int_{\S^3} f^3.
}
Therefore, by \eqref{eq:75} and \eqref{eq:93}, we have
\eq{\label{eq:77}
\frac{\int_{\S^3} f^{-1}\abs{\rd\alpha}^2}{\int_{\S^3} f^3}=
\frac{\int_{\S^3}\abs{\rd\alpha}^2_{\tilde{g}}\rdV_{\tilde{g}}}{\inf_\phi\int_{\S^3}\abs{\alpha+\rd\phi}^2_{\tilde{g}}\rdV_{\tilde{g}}}  \ge  \min\{\lambda_1^+(\tilde{g})^2,\lambda_1^-(\tilde{g})^2\} \ge \Big(2\omega_3^{\frac 13} \norm {f}_3^{-1}\Big)^2.
}
Using \eqref{eq:68} we have
\eq{\label{eq:78}
    \int_{\S^3} f\abs{S}^2 \geq \frac{3}{8}\int_{\S^3} f^{-1}\abs{\rd\alpha}^2
}
Combining \eqref{eq:75}, \eqref{eq:76}, \eqref{eq:77}, and \eqref{eq:78} yields
\eq{\label{eq:81}
    \int_{\S^3} f\abs{S}^2 \geq \frac{3}{8}\lambda_1^+(\tilde{g})^2\norm{f}_3^3 \geq \frac{3}{8}\Big( 2\omega_3^{\frac{1}{3}}\norm{f}_3^{-1} \Big)^2\norm{f}_3^3 = \frac{3}{2}\omega_3^{\frac{2}{3}}\norm{f}_3.
}
This finishes the proof of claim \eqref{eq:claim} and hence finishes the proof of
\eq{
    \norm{\curl A}_{\frac{3}{2}} \geq 3\omega_3^{\frac{2}{3}}.
}

From the proof above, one can also read off the equality case: equality holds if and only if
\begin{enumerate}
    \item equality holds in \eqref{eq:66}, i.e. \eqref{eq:82} holds;
    \item equality holds in \eqref{eq:73}, i.e. $i\rd A^\flat\cdot\phi=i(\curl A)\cdot\phi$ is pointwise proportional to $\phi$;
    \item equality holds in \eqref{eq:74}, i.e. $u$ is constant modulo conformal transformations;
    \item equality holds in \eqref{eq:80}, i.e. $f^3 = c\abs{\curl A}^{\frac{3}{2}}$ for some constant $c$;
    \item equality holds in \eqref{eq:75}, i.e. $\tilde g$ is round modulo conformal transformations;
    \item equality holds in \eqref{eq:93}, i.e. $\alpha$ is a first eigenform of $\curl$.
\end{enumerate}
Modulo conformal and gauge transformations, (3) and (5) imply that $f=\abs{\varphi}$ is constant, hence $\alpha=f\,\xi^\flat$ is proportional to $\xi^\flat$. Then (4) implies that $\abs{\curl A}$ is constant, and (6) implies that $\xi^\flat$ is a Killing $1$-form. Moreover, (2) and \eqref{eq:83} imply that $\curl A$ is proportional to $\xi$, hence an eigenfield of $\curl$, and so is $A$. The zero mode equation then shows that $\varphi$ is an eigenspinor of $\D$. Finally, (1) implies that $K=S=\mathrm{diag}\{-\tfrac{1}{2}S_{33},-\tfrac{1}{2}S_{33},S_{33}\}$. Together with \eqref{eq:64}, we conclude that (modulo conformal transformations) $\varphi$ is a Killing spinor and $A$ is a multiple of the Reeb field associated with $\varphi$. This gives the equality case in Theorem~\ref{thm:magnetic}.
\end{proof}

\subsection{The regularization}\label{sec:regularization}
We briefly justify the argument above when $\varphi$ has zeros. Since $f=\abs{\varphi}$ is Lipschitz, we have $\nabla f=0$ a.e. on the zero set $\{f=0\}$. Thus all pointwise identities extend trivially across $\{f=0\}$, except for terms that involve division by $f$ or the conformal metric $\tilde g=f^2g_{\rm st}$, namely
\eq{
    f\abs{S}^2, \qquad f^{-1}\abs{\rd\alpha}^2, \qquad u=\sqrt{f}, \qquad \tilde g=f^2g_{\rm st}.
}
We handle these by the  regularization
\eq{
    f_\epsilon\coloneqq \sqrt{\abs{\varphi}^2+\epsilon^2}=\sqrt{f^2+\epsilon^2},\qquad \epsilon>0
}
and follow the idea in \cite{Frank_Loss_2024}. Actually our case is even easier.
It is easy to see that  
\eq{\label{eq:85}
    \nabla^A_X(f_\epsilon^{-1}\varphi)= f_\epsilon^{-1}\nabla^A_X\varphi - f_\epsilon^{-3}f\,X(f)\,\varphi.
}
Using \eqref{eq:62}, \eqref{eq:85} and $\D^A\varphi=0$, and testing \eqref{eq:84a} against $f_\epsilon^{-1}\varphi$, we obtain
\eq{
    \int_{\S^3} f_\epsilon^{-1}\rRe\<i\rd A^\flat\cdot\varphi,\varphi\>
    = \frac{3}{2}\int_{\S^3} f_\epsilon^{-1}\abs{\nabla f}^2 + \int_{\S^3}f_\epsilon^{-1}f^2\abs{S}^2 - \int_{\S^3}f_\epsilon^{-3}f^2\abs{\nabla f}^2 + \frac{3}{2}\int_{\S^3}f_\epsilon^{-1}f^2.
}
Observe that
\eq{
    \frac{3}{2}f_\epsilon^{-1}\abs{\nabla f}^2 - f_\epsilon^{-3}f^2\abs{\nabla f}^2
    = \frac{3}{2}\epsilon^2f_\epsilon^{-3}\abs{\nabla f}^2 + \frac{1}{2}f_\epsilon^{-3}f^2\abs{\nabla f}^2,
}
so
\eq{\label{eq:86}
    \int_{\S^3} f_\epsilon^{-1}\rRe\<i\rd A\cdot\varphi,\varphi\>
    = \frac{3}{2}\int_{\S^3}\epsilon^2f_\epsilon^{-3}\abs{\nabla f}^2 + \frac{1}{2}\int_{\S^3}f_\epsilon^{-3}f^2\abs{\nabla f}^2 + \int_{\S^3}f_\epsilon^{-1}f^2\abs{S}^2 + \frac{3}{2}\int_{\S^3}f_\epsilon^{-1}f^2.
}
Moreover,
\eq{
    \int_{\S^3} f_\epsilon^{-1}\rRe\<i\rd A\cdot\varphi,\varphi\>
    \le \int_{\S^3} f\abs{\curl A}
    \le \norm{f}_3\,\norm{\curl A}_{\frac{3}{2}} < \infty,
}
and as $\epsilon\to0$,
\eq{
    f_\epsilon^{-3}f^2\abs{\nabla f}^2 \nearrow \mathbf{1}_{\{f>0\}}f^{-1}\abs{\nabla f}^2,\qquad
    f_\epsilon^{-1}f^2\abs{S}^2 \nearrow \mathbf{1}_{\{f>0\}}f\abs{S}^2,\qquad
    f_\epsilon^{-1}f^2 \nearrow f.
}
After dropping the nonnegative first term, by dominated (and monotone) convergence, all terms on the right-hand side of \eqref{eq:86} converge. Therefore
\eq{\label{eq:90}
    \norm{f}_3\,\norm{\curl A}_{\frac{3}{2}} \geq \frac{1}{2}\int_{\S^3}f^{-1}\abs{\nabla f}^2 + \frac{3}{2}\int_{\S^3}f + \int_{\S^3}f\abs{S}^2.
}
To estimate the first two integrals, set $f=u^2$ and use the critical Sobolev inequality:
\eq{\label{eq:91}
    \frac{1}{2}\int_{\S^3}f^{-1}\abs{\nabla f}^2 + \frac{3}{2}\int_{\S^3}f
    = 2\int_{\S^3}\abs{\nabla u}^2 + \frac{3}{2}\int_{\S^3}u^2
    \ge \frac{3}{2}\omega_3^{\frac{2}{3}}\Big(\int_{\S^3}u^6\Big)^{\frac{1}{3}}
    = \frac{3}{2}\omega_3^{\frac{2}{3}}\norm{f}_3.
}
For the last term in \eqref{eq:90}, define
\eq{
    \tilde g_\epsilon\coloneqq f_\epsilon^2g,
}
and set
\eq{
    \alpha\coloneqq f\xi^\flat\ \text{on }\{f>0\},\qquad \alpha\coloneqq 0\ \text{on }\{f=0\}.
}
Then $\rd\alpha=0$ a.e. on $\{f=0\}$, and \eqref{eq:77} becomes
\eq{
    \int_{\S^3} f_\epsilon^{-1}\abs{\rd\alpha}^2 \geq 4\,\omega_3^{\frac{2}{3}}\Big(\int_{\S^3} f_\epsilon^3\Big)^{-\frac{2}{3}}\cdot \inf_\phi\int_{\S^3}f_\epsilon\abs{\alpha+\rd\phi}^2,
}
where since $f_\epsilon\geq f$ and $\rd^*(f\alpha)=0$,
\eq{
    \inf_\phi\int_{\S^3}f_\epsilon\abs{\alpha+\rd\phi}^2 \geq \inf_\phi\int_{\S^3}f\abs{\alpha+\rd\phi}^2 = \int_{\S^3} f^3.
}
Hence
\eq{\label{eq:87}
    \int_{\S^3} f_\epsilon^{-1}\abs{\rd\alpha}^2 \geq 4\,\omega_3^{\frac{2}{3}}\Big(\int_{\S^3} f_\epsilon^3\Big)^{-\frac{2}{3}}\cdot \int_{\S^3} f^3.
}
Using \eqref{eq:68}, we have
\eq{\label{eq:88}
    f_\epsilon^{-1}\abs{\rd\alpha}^2 \nearrow \mathbf{1}_{\{f>0\}}f^{-1}\abs{\rd\alpha}^2
    \le \mathbf{1}_{\{f>0\}}\cdot \frac{8}{3}f\abs{S}^2
    \qquad\text{as }\epsilon\to0,
}
while
\eq{\label{eq:89}
    \Big(\int_{\S^3} f_\epsilon^3\Big)^{\frac{1}{3}} \to \norm{f}_3\qquad\text{as }\epsilon\to0.
}
Combining \eqref{eq:87} and \eqref{eq:89} gives
\eq{\label{eq:92}
    \int_{\S^3} f\abs{S}^2 \geq \frac{3}{2}\,\omega_3^{\frac{2}{3}}\norm{f}_3.
}
Finally, \eqref{eq:90}, \eqref{eq:91}, and \eqref{eq:92} imply
\eq{
    \norm{\curl A}_{\frac{3}{2}} \geq 3\,\omega_3^{\frac{2}{3}},
}
which is the desired inequality.

Since the proof uses an approximation argument, the classification of the equality case requires some care. For details, we refer the interested reader to \cite{Frank_Loss_2024}.

\appendix
\section{Weighted Yamabe problems}\label{subsection3.2}

Theorem~\ref{thm:GN-Sn} is closely related to \cite{Case15}, where Case introduced the weighted Yamabe quotient on a closed smooth metric measure space $(M^n,g,e^{-\phi}\rdV_g,m)$. For $m\geq0$ and $0<f\in C^\infty(M)$, set
\eq{
    Q_{g,\phi,m}(f) \coloneqq
    \frac{
        \Big( \int_M \bigl(\abs{\nabla f}^2 + \tfrac{m+n-2}{4(m+n-1)}R^m_\phi f^2\bigr)\, e^{-\phi}\rdV_g \Big)
        \Big( \int_M \abs{f}^{\frac{2(m+n-1)}{m+n-2}}e^{\frac{\phi}{m}}\, e^{-\phi}\rdV_g \Big)^{\frac{2m}{n}}
    }{
        \Big( \int_M \abs{f}^{\frac{2(m+n)}{m+n-2}}\, e^{-\phi}\rdV_g \Big)^{\frac{2m+n-2}{n}}
    },
}
where $R^m_\phi$ is the weighted scalar curvature.
This quotient is weighted conformally invariant in the sense that
\eq{
    Q_{\tilde{g},\tilde{\phi},m}(\tilde{f}) = Q_{g,\phi,m}(f),
}
where
\eq{
    \tilde{g}=\rho^2g, \qquad \tilde{\phi}=\phi-m\log\rho, \qquad \tilde{f}=\rho^{-\frac{m+n-2}{2}}f.
}

The weighted Yamabe constant is defined by
\eq{
    \Lambda(M,[g,\phi,m]) \coloneqq \inf\Bigl\{Q_{g,\phi,m}(f) \,\Big|\, 0<f\in C^\infty(M)\Bigr\}.
}
Here $[g,\phi,m]$ denotes the weighted conformal class in the above sense. When $m=0$ one recovers the usual Yamabe constant (interpreting $Q_{g,\phi,m}$ in the appropriate limiting sense). In our setting ($M=\S^n$, $m=1$, $\phi=0$) the quotient specializes to \eqref{eq:GN-Sn}.

Case proved (\cite[Theorem 1.2]{Case15}) that
\eq{\label{eq:Case}
    \Lambda(M,[g,\phi,m]) \leq \Lambda(\R^n,[g_{\R^n},0,m]).
}
Moreover, \cite[Theorem 1.3]{Case15} shows that if $m\in\mathbb{N}\cup\{0\}\cup\{\infty\}$, then equality in \eqref{eq:Case} forces
\eq{
    m\in\{0,1\}, \quad (M,[g,\phi,m])=(\S^n,[g_{{\rm st}},0,m]).
}
Thus, for such $m$, Case obtained a necessary condition for equality. When $m=1$, our inequality \eqref{eq:GN-Sn} implies
\[
\Lambda(\S^n,[g_{{\rm st}},0,1])=\Lambda(\R^n,[g_{\R^n},0,1]),
\]
which provides the corresponding sufficiency statement on the sphere. In contrast, for integers $m\ge 2$ it was proved in \cite{Case15} that
\[
\Lambda(\S^n,[g_{{\rm st}},0,m])<\Lambda(\R^n,[g_{\R^n},0,m]).
\]
Case also showed that
\[
\Lambda(\S^n,[g_{{\rm st}},0,\tfrac12])=\Lambda(\R^n,[g_{\R^n},0,\tfrac12]).
\]
It is natural to ask whether
\[
\Lambda(\S^n,[g_{{\rm st}},0,m])=\Lambda(\R^n,[g_{\R^n},0,m]) \Longleftrightarrow m\le 1.
\]
Here $\Lambda(\R^n,[g_{\R^n},0,m])$ is the sharp constant obtained by Del~Pino--Dolbeault \cite{DelPino02}.

\section{Sharp lower bounds of the original Faddeev-Skyrme model}

In this appendix, we give further applications to 
sharp lower bounds for the original Faddeev--Skyrme model for fields 
$u:\R^3\to\S^2$. We also recall that $\omega_3=\rVol(\S^3)=2\pi^2$.

\subsection{The Faddeev--Skyrme model}
The original Faddeev--Skyrme model \cite{Faddeev} is defined on $\R^3$ by
\eq{
    \mathcal{FS}_{{\rm Euc}}(u) \coloneqq \int_{\R^3}\Big( \abs{\rd u}^2 + \frac{1}{2}\abs{u^*\omega_{\S^2}}^2 \Big),
    \qquad u:\R^3\to\S^2\ \text{smooth}.
}
It satisfies the Vakulenko--Kapitansky bound \cite{VK79}
\eq{\label{eq:VK}
    \mathcal{FS}_{{\rm Euc}}(u) \geq C_{\text{VK}}\,\abs{Q(u)}^{\frac{3}{4}},
}
where $C_{{\rm VK}}>0$ is a universal constant and $Q$ is the Hopf invariant
\eq{
    Q(u) \coloneqq \frac{1}{16\pi^2}\int_{\R^3}\alpha\w\rd\alpha,
}
with $\rd\alpha=u^*\omega_{\S^2}$ for some $\alpha\in\Omega^1(\R^3)$. 
For finite energy fields, $Q$ is well-defined and is a homotopy invariant. We only consider $Q\neq0$.
The canonical Hopf map $\pi\circ\Psi$ satisfies $Q(\pi\circ\Psi)=1$, where 
$\Psi:\R^3\to\S^3$ denotes the inverse stereographic projection.

The sublinear growth in \eqref{eq:VK} is one of the key features of the model. 
In the original arguments it arises from delicate Sobolev inequalities; 
below we explain how it fits naturally into the curl--Sobolev framework. 
Lin--Yang \cite{LinYang04} later proved a positive (non-sharp) lower bound, 
confirming that the energy grows sublinearly as $\abs{Q(u)}\to\infty$. 
From the physics perspective, this behavior is commonly attributed to the 
creation of additional charge through knotting and linking of solitons.

Inequality \eqref{eq:VK} was improved in \cite{KunduRybakov82} to the explicit estimate
\eq{\label{eq:124}
    C_{\text{VK}}\geq 3^{\frac{3}{8}}\,8\sqrt{2}\,\pi^2,
}
and Ward conjectured in \cite{Ward99} that in fact
\eq{\label{eq:ward_conj}
    C_{\text{VK}} \geq 32\pi^2.
}
See also \cite{Harland13MassivePions}. Numerical simulations suggest that minimizers 
exceed Ward's conjectured bound by about $20\%$; see \cite{Ward99,Sutcliffe07}.

Following \cite{LinYang04}, we define
\eq{
E_m\coloneqq \inf\{\mathcal{FS}_{\text{Euc}}(u)\,|\, 
\mathcal{FS}_{\text{Euc}}(u)<\infty, \, Q(u)=m\}.
}
Ward \cite{Ward99} also proved that, in the homotopy class of $\pi\circ\Psi$ (i.e.
$Q=1$), one has the upper bound
\eq{\label{eq:116}
    E_1 \leq 32\sqrt{2}\pi^2.
}
Here $E_1$ follows the notation of \cite{LinYang04}; see \cite{LinYang07} for a clean proof.
Moreover, Lin--Yang \cite{LinYang07} showed --using \eqref{eq:124} and \eqref{eq:116} together with 
their inequality from \cite{LinYang04}-- that $E_1$ is attained.

As an application of Theorem~\ref{main_thm}, we confirm Ward's conjecture.

\begin{theorem}\label{propB.1}
Ward's conjecture is true. Namely,
\eq{\label{eq:123}
    \mathcal{FS}_{{\rm Euc}}(u) > 32\pi^2\abs{Q(u)}^{\frac{3}{4}}.
} In particular $E_{m} \ge 32\pi^2|m|^{\frac 34}$.
Moreover, we have
\eq{E_1 >32\pi^2.}
\end{theorem}

\begin{proof}
Using the pointwise estimate $\abs{\rd u}^2\geq 2\abs{\rd\alpha}$ and the AM--GM inequality,
\eq{\label{eq:117}
    \mathcal{FS}_{{\rm Euc}}(u)
    \geq \int_{\R^3}\Bigl( 2\abs{\rd\alpha} + \tfrac12\abs{\rd\alpha}^2 \Bigr)
    \geq 2\int_{\R^3}\abs{\rd\alpha}^{\frac{3}{2}}
    = 2\int_{\S^3}\abs{\rd\tilde{\alpha}}^{\frac{3}{2}},
}
where $\tilde{\alpha}=(\Psi^{-1})^*\alpha$, and in the last step we used conformal invariance.
By Theorem~\ref{main_thm},
\eq{\label{eq:118}
    \int_{\S^3}\abs{\rd\tilde{\alpha}}^{\frac{3}{2}}
    \geq \Bigl(2\,\omega_3^{\frac{1}{3}}\Abs{\int_{\S^3}\tilde{\alpha}\w\rd\tilde{\alpha}}\Bigr)^{\frac{3}{4}}.
}
Combining \eqref{eq:117} and \eqref{eq:118} yields
\eq{
    \mathcal{FS}_{{\rm Euc}}(u)
    \geq 2^{\frac{7}{4}}\omega_3^{\frac{1}{4}}\Abs{\int_{\S^3}\tilde{\alpha}\w\rd\tilde{\alpha}}^{\frac{3}{4}}
    = 32\pi^2\abs{Q(u)}^{\frac{3}{4}}.
}
If equality held, then in particular $\abs{\rd\alpha}\equiv 4$, which is impossible on $\R^3$
by integrability considerations. Hence the inequality is strict.

Together with the attainability of $E_1$ proved by Lin--Yang \cite{LinYang07}, we obtain
\eq{\label{eq:126}
    E_1 > 32\pi^2.
}
\end{proof}
Thus, for $Q=1$ the constant $32\pi^2$ is not the sharp lower bound within that homotopy class.
Nevertheless, we expect $32\pi^2$ to be optimal when taking the infimum over all fields and all $Q$.

As a direct application, we answer a question left in \cite{LinYang04, LinYang07}.

\begin{corollary} Let $|m|\ge 2$, we have \eq{E_1<E_m.\label{eq:Em}}    
\end{corollary}
\begin{proof} For $|m|>3$, \eqref{eq:Em} was proved in \cite{LinYang04}. Now the sharp lower bound \eqref{eq:123}, together with the upper bound \eqref{eq:116}, yields
\eq{
E_1 \le 32\sqrt {2} \pi^2 < 32 \pi^2 2^{\frac 34} \le 32\pi^2 |m|^\frac 34 
\le E_m.}   
\end{proof}

One can also improve the upper bound \eqref{eq:116}; in particular,\eq{\label{eq:119}    E_1 \leq 8\sqrt{30}\pi^2 < 32\sqrt{2}\pi^2.}We omit the proof.

To determine the value of $E_1$ is a very difficult problem.

\subsection{Modified Faddeev--Skyrme models}

There are various modified models for the Faddeev--Skyrme model. One of the motivations is that there is no analytic solutions in the Faddeev--Skyrme model. Here we discuss a little about the following three models.
\begin{enumerate}
  
\item 
 The Nicole model \cite{Nicole78} is defined
by
\eq{
    E_{{\rm Ni}}(u) \coloneqq \frac 1{\sqrt 2}\int_{\R^3} \abs{\rd u}^3, \quad u:\R^3\to\S^2.
}
  \item 
The Aratyn--Ferreira--Zimerman (AFZ) model \cite{AFZ99} is defined by 
\eq{
    E_{{\rm AFZ}}(u) \coloneqq 2\int_{\R^3} \abs{u^*\omega_{\S^2}}^{\frac{3}{2}}, \quad u:\R^3\to\S^2.
}

\item The conformal Faddeev--Skyrme model defined by Gillard \cite{Gillard10}, 
which interpolates between the AFZ and Nicole models:
\eq{
    E_{{\rm CFS}}(u) \coloneqq \cos^2\theta\,E_{{\rm Ni}}(u) + \sin^2\theta\,E_{{\rm AFZ}}(u),
    \qquad u:\R^3\to\S^2,\quad 0\leq\theta\leq\frac{\pi}{2}.
}
\end{enumerate}
The normalizations (the factor $2$ in (1) and $\frac{1}{\sqrt 2}$ in (2)) are chosen so that
all three models share the same sharp lower bound in Theorem~\ref{thmB.6}.

Unlike the original Faddeev--Skyrme model, these modified models admit explicit analytic solutions;
in particular, the AFZ model even has infinitely many solutions.

In contrast to the Faddeev--Skyrme model, until now there had been no established 
Vakulenko--Kapitansky type bound for these conformally invariant variants.
If such a bound exists, it was conjectured that the optimal constant equals $32\pi^2$
for both the Nicole model \cite{Nicole78} and the AFZ model \cite{AFZ99}; see also \cites{Gillard10,Sutcliffe07}.
Our main result implies the existence of such a bound and confirms this conjecture.
Since the conformal Faddeev--Skyrme model contains the other two as special cases,
we state the result in that generality.

\begin{theorem}\label{thmB.6}
For every $0\leq\theta\leq\frac{\pi}{2}$ and every smooth $u:\R^3\to\S^2$,
\eq{
    E_{{\rm CFS}}(u) \geq 32\pi^2\abs{Q(u)}^{\frac{3}{4}}.
}
Moreover, equality holds if and only if $u=h\circ\pi\circ\Phi\circ\Psi$, where
$\Phi\in{\rm Conf}(\S^3)$, $\Psi:\R^3\to\S^3$ is the inverse stereographic projection,
and $h:\S^2\to\S^2$ satisfies either $h^*\omega_{\S^2}=\omega_{\S^2}$ (when $\theta=\frac{\pi}{2}$),
or $h$ is an isometry (when $0\leq\theta<\frac{\pi}{2}$).
\end{theorem}

\begin{proof} For any $u:\R^3 \to \S^2$, let $\tilde u =u\circ \Psi^{-1}:\S^3\to \S^2$. Moreover,  let $\alpha\in\Omega^1(\R^3) $ satisfy $\rd \alpha =u^\ast \omega_{\S^2}$ and $\tilde \alpha=(\Psi^{-1})^*\alpha$. It is clear that $\rd\tilde \alpha=\tilde{u}^*\omega_{\S^2}$.

\smallskip

\noindent\textit{Case 1: $\theta=\frac{\pi}{2}$ (the AFZ model).}
By conformal invariance,
\eq{\label{eq:120}
    E_{{\rm AFZ}}(u) = 2\int_{\R^3}\abs{\rd{\alpha}}^{\frac{3}{2}} = 2\int_{\S^3}\abs{\rd\tilde \alpha}^{\frac{3}{2}}.
}
Theorem~\ref{main_thm} yields
\eq{\label{eq:121}
    \int_{\S^3}\abs{\rd\tilde \alpha}^{\frac{3}{2}}
    \geq 
\Bigl(2\omega_3^{\frac{1}{3}}\Abs{\int_{\S^3} \tilde \alpha\w\rd\tilde \alpha}\Bigr)^{\frac{3}{4}}=\Bigl(2\omega_3^{\frac{1}{3}}\Abs{\int_{\R^3} \alpha\w\rd\alpha}\Bigr)^{\frac{3}{4}}
    = 16\pi^2\abs{Q(u)}^{\frac{3}{4}}.
}
Combining \eqref{eq:120} and \eqref{eq:121} gives
\eq{
    E_{{\rm AFZ}}(u) \geq 32\pi^2\abs{Q(u)}^{\frac{3}{4}}.
}
As in the proof of Theorem~\ref{thm:hopf-3energy-intro}, one can characterize equality;
here $h$ is an area-preserving diffeomorphism.

\smallskip

\noindent\textit{Case 2: $\theta=0$ (the Nicole model).}
By conformal invariance,
\eq{\label{eq:100}
    E_{{\rm Ni}}({u})
    = \frac{1}{\sqrt{2}}\int_{\R^3}\abs{\rd {u}}^3
    = \frac{1}{\sqrt{2}}\int_{\S^3}\abs{\rd \tilde u}^3,
}
This is exactly the situation considered in Section~\ref{app:3-energy}.
Using the pointwise estimate
\eq{\label{eq:23a}
    \abs{\rd \tilde u}^2 \geq 2\abs{\rd\tilde \alpha},
}
we obtain the same lower bound as in Case~1.
In the equality case, one must also have equality in \eqref{eq:23a}, which forces $h$ to be an isometry.

\smallskip

\noindent\textit{Case 3: $\theta\in(0,\frac{\pi}{2})$.}
This follows by combining the previous two cases.
\end{proof}

\section{Regularity of solutions}\label{app:C}
We study the regularity of weak solutions of
\eq{\label{C1}
    \curl \alpha = \frac{n+1}{2}\abs{\alpha}^{\frac{2}{n-1}}\alpha.
}
The following statement is standard (and should be well known to experts); we include a sketch for the reader's convenience.

\begin{theorem}\label{thmC.1}
Let $\alpha\in W^{1, \frac{2n}{n+1}}\big(\Omega^{\frac{n-1}{2}}(\S^n)\big)$ be a weak solution of \eqref{C1}. Then $\alpha\in W^{1,2}(\S^n)$ and, moreover, $\alpha\in C^{1,\gamma}$ on $\{\alpha\neq 0\}$ for some $\gamma>0$.
\end{theorem}

\begin{proof}

\noindent\emph{Step 1: H\"older continuity of $\alpha$.}
Fix a small ball $B_r\subset \S^n$. 
The averaged Poincar\'e homotopy estimate (see \cite[Proposition 4.1]{IL93}) implies that there exists a bounded linear operator $T_r$ such that
\eq{
    \alpha = T_r(\rd\alpha) + \rd(T_r\alpha) \eqcolon \beta_r + \rd\phi_r,
}
and
\eq{
    \norm{\beta_r}_{L^{\frac{2n}{n-1}}(B_r)} \lesssim \norm{\rd\alpha}_{L^{\frac{2n}{n+1}}(B_r)}
    = \norm{\alpha}_{L^{\frac{2n}{n-1}}(B_r)}^{\frac{n+1}{n-1}}.
}
Let $h=\rd\psi$ be the $p$-harmonic replacement of $\rd\phi_r$ in $B_r$, with $p=\frac{2n}{n-1}$; namely, $\psi$ minimizes $\int_{B_r}\abs{\rd\psi}^{p}$ among competitors with $\psi-\phi_r\in W^{1,p}_0(B_r)$. Then
\eq{
    \rd h = 0, \qquad \rd^*\big( \abs{h}^{\frac{2}{n-1}}h \big) = 0,
}
and the interior estimate yields (see, e.g., \cite[Theorem~4.1]{Hamburger92})
\eq{\label{C2}
    \sup_{B_{r/2}} \abs{h}^{\frac{2n}{n-1}} \lesssim r^{-n}\int_{B_r}\abs{h}^{\frac{2n}{n-1}}.
}
Testing 
\eq{
    \rd^*\big( \abs{\alpha}^{\frac{2}{n-1}}\alpha \big) = 0, \qquad
    \rd^*\big( \abs{h}^{\frac{2}{n-1}}h \big) = 0
}
against $\phi-\psi$ and subtracting, we obtain
\eq{
    \int_{B_r} \<\abs{\alpha}^{\frac{2}{n-1}}\alpha - \abs{h}^{\frac{2}{n-1}}h, \alpha-h\>
    = \int_{B_r}\<\abs{\alpha}^{\frac{2}{n-1}}\alpha - \abs{h}^{\frac{2}{n-1}}h, \beta\>.
}
Using the standard monotonicity and growth bounds
\eq{
    \<\abs{\alpha}^{\frac{2}{n-1}}\alpha - \abs{h}^{\frac{2}{n-1}}h, \alpha-h\> \geq c_n\abs{\alpha-h}^{\frac{2n}{n-1}},
}
\eq{
    \Abs{\abs{\alpha}^{\frac{2}{n-1}}\alpha - \abs{h}^{\frac{2}{n-1}}h}
    \leq C_n\big( \abs{\alpha}^{\frac{n+1}{n-1}} + \abs{h}^{\frac{n+1}{n-1}} \big),
} we obtain, by
 H\"older's inequality, 
\eq{\label{C3}
    \int_{B_r}\abs{\alpha-h}^{\frac{2n}{n-1}}
    \lesssim \Big( \norm{\alpha}_{\frac{2n}{n-1}}^{\frac{n+1}{n-1}} + \norm{h}_{\frac{2n}{n-1}}^{\frac{n+1}{n-1}} \Big)
    \norm{\beta_r}_{\frac{2n}{n-1}}.
}
For $r$ sufficiently small we may assume $\int_{B_r}\abs{\alpha}^{\frac{2n}{n-1}}\leq \epsilon$ (with $\epsilon$ to be fixed).  By minimality of $h$ we  have $\norm{h}_{\frac{2n}{n-1}}\lesssim \norm{\alpha}_{\frac{2n}{n-1}}$ on $B_r$, and hence
\eq{
    \int_{B_r}\abs{\alpha-h}^{\frac{2n}{n-1}} \lesssim \Big( \int_{B_r}\abs{\alpha}^{\frac{2n}{n-1}} \Big)^{\frac{n+1}{n}} \leq \epsilon\Big(\int_{B_r}\abs{\alpha}^{\frac{2n}{n-1}} \Big)^{\frac{1}{n}}.
}
Define
\eq{
    \Psi_x(s) \coloneqq \int_{B_s(x)}\abs{\alpha}^{\frac{2n}{n-1}}.
}
Using \eqref{C2},  for $\theta\in\big(0,\frac12\big)$ we get 
\eq{
    \int_{B_{\theta r}(x)}\abs{h}^{\frac{2n}{n-1}} \lesssim \theta^n \Psi_x(r).
}
Combining the last two bounds yields
\eq{
    \Psi_x(\theta r) \leq C\big( \theta^n + \Psi_x(r)^{\frac{1}{n}} \big)\Psi_x(r).
}
Fix $\mu<n$. Choose $\theta\in\big(0,\frac12\big)$ such that $C\theta^n\leq\frac12\theta^\mu$, and then choose $\epsilon>0$ such that $C\epsilon^{1/n}\leq\frac12\theta^\mu$. Then
\eq{\label{C4}
    \Psi_x(\theta r) \leq \theta^\mu\Psi_x(r).
}
Iterating \eqref{C4} yields
\eq{
    \int_{B_r(x)}\abs{\alpha}^{\frac{2n}{n-1}} \leq C(\mu,\alpha)\, r^\mu, \qquad x\in\S^n.
}
We may now apply Uhlenbeck's estimate (see, e.g., \cite[Theorem~4.1]{Hamburger92}) to deduce that
\eq{
    \inf_c \int_{B_r(x)}\abs{\alpha-c}^{\frac{2n}{n-1}} \lesssim r^{n+\frac{2n}{n-1}\sigma}
}
for some $\sigma>0$. Campanato's theorem \cite{Campanato63} then implies $\alpha\in C^{0,\gamma}$ for some $\gamma>0$.

\smallskip
\noindent\emph{Step 2: the $W^{1,2}$ estimate.}
Since $\alpha\in C^{0,\gamma}$, the right-hand side of \eqref{C1} is H\"older, and hence so is $\rd\alpha$. The form $\beta$ satisfies the elliptic system
\eq{
    \rd\beta = \rd\alpha\in C^{0,\gamma}, \qquad \rd^*\beta=0,
}
so standard Schauder theory implies $\beta\in C^{1,\gamma}\subset W^{1,2}$.

To estimate $\rd\phi$ we use a regularization argument. Define
\eq{
    J_\epsilon(\psi) \coloneqq \frac{n-1}{2n}\int_{\S^n}\big( \epsilon^2 + \abs{\beta+\rd\psi}^2 \big)^{\frac{n}{n-1}},
}
and let $\phi_\epsilon$ be a minimizer subject to $\rd^*\phi_\epsilon=0$ (when $n=3$, subject to $\int_{\S^3}\phi_\epsilon=0$). Set $\alpha_\epsilon\coloneqq \beta+\rd\phi_\epsilon$. The Euler--Lagrange equation is 
\eq{\label{C5}
    \rd^*\Big( (\epsilon^2+\abs{\alpha_\epsilon}^2)^{\frac{1}{n-1}}\alpha_\epsilon \Big) = 0.
}
Since $\rd^*\beta=0$, expanding \eqref{C5} yields
\eq{\label{C6}
    \abs{\rd^*\rd\phi_\epsilon}
    \leq \frac{2}{n-1}\abs{\nabla\alpha_\epsilon}
    \leq \frac{{2}}{n-1}\big(\abs{\nabla\rd\phi_\epsilon}+\abs{\nabla\beta}\big).
}
On $\S^n$, the Weitzenb\"ock formula gives
\eq{\label{C7}
    \norm{\nabla\rd\phi_\epsilon}_2^2 + \frac{(n+1)(n-1)}{4}\norm{\rd\phi_\epsilon}_2^2
    = \norm{\rd^*\rd\phi_\epsilon}_2^2.
}
Combining \eqref{C6} and \eqref{C7} we obtain
\eq{
    \norm{\nabla\rd\phi_\epsilon}_2
    \leq \frac{2}{n-1}\norm{\nabla\rd\phi_\epsilon}_2 + \frac{2}{n-1}\norm{\nabla\beta}_2.
}
If $n\geq 5$ this implies
\eq{
    \norm{\nabla\rd\phi_\epsilon}_2 \leq \frac{2}{n-3}\norm{\nabla\beta}_2,
}
which implies  
$\rd\phi\in W^{1,2}$ and hence $\alpha \in W^{1,2}$.

If $n=3$, then \eqref{C5} becomes
\eq{
    \rd^*\Big( (\epsilon^2+\abs{\alpha_\epsilon}^2)^{\frac{1}{2}}\alpha_\epsilon \Big) = 0,
}
where $\alpha_\epsilon$ is a $1$-form. A direct algebraic computation shows
\eq{\label{C8}
    \abs{\rd^*\rd\phi_\epsilon} = \abs{\rd^*\alpha_\epsilon} \leq \frac{1}{\sqrt{3}}\abs{\nabla\alpha_\epsilon}.
}
Replacing \eqref{C6} by \eqref{C8} and repeating the same argument gives again $\alpha\in W^{1,2}$.

\smallskip
\noindent\emph{Step 3: higher regularity away from the zero set.}
On $\{\alpha\neq0\}$ the system 
\eq{
    \rd\alpha = \frac{n+1}{2}\abs{\alpha}^{\frac{2}{n-1}}*\alpha, \qquad \rd^*\big( \abs{\alpha}^{\frac{2}{n-1}}\alpha \big) = 0
}
is uniformly elliptic, so standard elliptic regularity implies $\alpha\in C^{1,\gamma}$ (indeed, $C^\infty$) there.
\end{proof}

\medskip

\noindent{\it Acknowledgements.} The authors thank Rupert Frank and Derek Harland for their interest in this work and for helpful communication regarding the Faddeev--Skyrme model.

\medskip 

\noindent\textsc{Data availability.}
No new data were created or analyzed in this study; data sharing is not applicable.

\medskip

\noindent\textsc{Competing interests.}
The authors declare that they have no competing interests.

\printbibliography

@article{B19,
  title={The curl operator on odd-dimensional manifolds},
  author={B{\"a}r, C.},
  journal={Journal of Mathematical Physics},
  shortjournal={J. Math. Phys.},
  volume={60},
  number={3},
  year={2019},
  publisher={AIP Publishing}
}

@article{Montiel23IsoCurl,
  title={The Isoperimetric Problem for the Curl Operator},
  author={Montiel, S.},
  eprint={2307.09556},
  year={2023},
  archivePrefix= {arXiv}
}

@article{BCV26WeylCurl,
  title={Higher order Weyl coefficients for the operator curl},
  author={Bracchi, G. and Capoferri, M. and Vassiliev, D.},
  eprint={2607.03273},
  year={2026},
  archivePrefix= {arXiv}
}

@article{CDGT00,
  title={Isoperimetric problems for the helicity of vector fields and the {B}iot--{S}avart and curl operators},
  author={Cantarella, J. and DeTurck, D. and Gluck, H. and Teytel, M.},
  journal={Journal of Mathematical Physics},
  shortjournal={J. Math. Phys.},
  volume={41},
  number={8},
  pages={5615--5641},
  year={2000},
  publisher={American Institute of Physics}
}

@article{FL22,
  title        = {Existence of optimizers in a {S}obolev inequality for vector fields},
  author       = {Frank, R. L. and Loss, M.},
  journal      = {Ars Inven. Anal.},
  year         = {2022},
  number       = {1},
  publisher    = {Ars Inveniendi}
}

@article{Figalli_Zhang_20,
  title={Sharp gradient stability for the {S}obolev inequality},
  author={Figalli, A. and Zhang, Y. R. Y.},
  journal={Duke Math. J. },
  volume={171},
  number={12},
  pages={2407--2459},
  year={2022},
  publisher={Duke University Press}
}

@article{WZ25b,
  title={Conformal invariants for the zero mode equation},
  author={Wang, G. and Zhang, M.},
  year         = {2025},
  eprint       = {2512.17854},
  archivePrefix= {arXiv}
}

@article{Z26,
  title={A conformal lower bound of weighted {D}irac eigenvalues on manifolds with boundary},
  author={Zhang, M.},
  year={2026},
  eprint       = {2603.10875},
  archivePrefix= {arXiv}
}

@article{MS25,
  title={Optimizers in {S}obolev-curl inequalities},
  author={Mederski, J. and Szulkin, A.},
  eprint={2511.01432},
  year={2025},
  archivePrefix= {arXiv}
}

@article{WaZh26,
  author = {Wang, G. and Zhang, M.},
  title  = {On the sharp constants in curl-Sobolev inequalities on $\mathbb{S}^n$},
  eprint = {2607.19091},
  year = {2026},
  archivePrefix= {arXiv}
}

@article{ColboisElSoufi06,
  author  = {Colbois, B. and El Soufi, A.},
  title   = {Eigenvalues of the {L}aplacian acting on $p$-forms and metric conformal deformations},
  journal = {Proceedings of the American Mathematical Society},
  shortjournal = {Proc. Amer. Math. Soc.},
  volume  = {134},
  year    = {2006},
  pages   = {715--721}
}

@article{Hersch70,
  author  = {Hersch, J.},
  title   = {Quatre propri{\'{e}}t{\'{e}}s isop{\'{e}}rim{\'{e}}triques de membranes sph{\'{e}}riques homog{\`{e}}nes},
  journal = {C. R. Acad. Sci. Paris S{\'{e}}r. A--B},
  volume  = {270},
  year    = {1970},
  pages   = {1645--1648}
}

@article{ElSoufiIlias86,
  author  = {El Soufi, A. and Ilias, S.},
  title   = {Immersions minimales, premi{\`{e}}re valeur propre du laplacien et volume conforme},
  journal = {Mathematische Annalen},
  shortjournal = {Math. Ann.},
  volume  = {275},
  number  = {2},
  year    = {1986},
  pages   = {257--267}
}

@article{FL1,
  title        = {Which magnetic fields support a zero mode?},
  author       = {Frank, R. L. and Loss, M.},
  journal      = {J. Reine Angew. Math.},
  volume       = {788},
  pages        = {1--36},
  year         = {2022},
  publisher    = {De Gruyter}
}

@book{Lawson,
  title={Spin Geometry},
  author={Lawson, H. B. and Michelsohn, M. L.},
  isbn={9780691085425},
%  lccn={89032544},
  series={Princeton Mathematical Series},
 % url={https://books.google.de/books?id=EWX1vgEACAAJ},
  year={1989},
  publisher={Princeton University Press}
}

@article{EGP25,
  title={Optimal metrics for the first curl eigenvalue on 3-manifolds},
  author={Enciso, A. and Gerner, W. and Peralta-Salas, D.},
  journal={Calculus of Variations and Partial Differential Equations},
  shortjournal={Calc. Var. PDEs},
  volume={64},
  number={146},
  year={2025},
  publisher={Springer}
}

@article{MarquesNevesWillmore2014,
  title={Min-max theory and the {W}illmore conjecture},
  author={Marques, F. C. and Neves, A.},
  journal={Annals of mathematics},
  shortjournal={Ann. of Math},
  volume={179},
  number={2},
  pages={683--782},
  year={2014},
  publisher={JSTOR}
}

@article{Loss_Yau_86,
  title={Stability of {C}oulomb systems with magnetic fields: {III}. {Z}ero energy bound states of the {P}auli operator},
  author={Loss, M. and Yau, H. T.},
  journal={Communications in mathematical physics},
  shortjournal={Comm. Math. Phys.},
  volume={104},
  number={2},
  pages={283--290},
  year={1986},
  publisher={Springer}
}

@article{Reuss25,
  title={A note on the existence of nontrivial zero modes on {R}iemannian manifolds},
  author={Reu{\ss}, Jonah},
year={2025},
  eprint       = {2503.01602},
  archivePrefix= {arXiv}
}

@article{Julio26,
  title={Spinor inequality for magnetic fields on spin manifolds},
  author={Julio-Batalla, J.},
  eprint={2603.23218},
  year={2026},
  archivePrefix= {arXiv}
}

@article{Frank_Loss_2024,
  title={A sharp criterion for zero modes of the {D}irac equation},
  author={Frank, R. L. and Loss, M.},
  journal={Journal of the European Mathematical Society},
  shortjournal={J. Eur. Math. Soc. (JEMS)},
  volume={28},
  number={1},
  pages={305--331},
  year={2026}
}

@article{Jammes07,
  title={Minoration conforme du spectre du laplacien de {H}odge-de {R}ham},
  author={Jammes, P.},
  journal={Manuscripta Math.},
  volume={123},
  number={1},
  pages={15--23},
  year={2007},
  publisher={Springer}
}

@article{EGP22,
  title={Optimal convex domains for the first curl eigenvalue in dimension three},
  author={Enciso, A. and Gerner, W. and Peralta-Salas, D.},
  journal={Transactions of the American Mathematical Society},
  shortjournal={Trans. Amer. Math. Soc.},
  volume={377},
  number={07},
  pages={4519--4540},
  year={2024}
}

@article{Riviere98CAG,
  title        = {Minimizing fibrations and $p$-harmonic maps in homotopy classes from $\mathbb{S}^3$ into $\mathbb{S}^2$},
  author       = {Rivi{\`e}re, T.},
  journal      = {Communications in Analysis and Geometry},
  shortjournal = {Comm. Anal. Geom.},
  volume       = {6},
  number       = {3},
  pages        = {427--483},
  year         = {1998}
}

@article{Isobe08,
  title={On a minimizing property of the {H}opf soliton in the {F}addeev--{S}kyrme model},
  author={Isobe, T.},
  journal={Reviews in Mathematical Physics},
  shortjournal={Rev. Math. Phys.},
  volume={20},
  number={07},
  pages={765--786},
  year={2008},
  publisher={World Scientific}
}

@article{VK79,
  title        = {Stability of solitons in $S^2$ in the nonlinear $\sigma$-model},
  author       = {Vakulenko, A. F. and Kapitanski, L. V.},
  journal      = {Dokl. Akad. Nauk SSSR},
  volume       = {246},
  number       = {4},
  pages        = {840--842},
  year         = {1979}
}

@article{Esteban86,
  title        = {A direct variational approach to {S}kyrme's model for meson fields},
  author       = {Esteban, M. J.},
  journal      = {Comm. Math. Phys.},
  volume       = {105},
  pages        = {571--591},
  year         = {1986}
}

@article{LinYang04,
  title        = {Existence of energy minimizers as stable knotted solitons in the {F}addeev model},
  author       = {Lin, F. and Yang, Y.},
  journal      = {Comm. Math. Phys.},
  volume       = {249},
  pages        = {273--303},
  year         = {2004}
}

@article{DelPino02,
  title={Best constants for {G}agliardo--{N}irenberg inequalities and applications to nonlinear diffusions},
  author={Del Pino, M. and Dolbeault, J.},
  journal={Journal de Math{\'e}matiques Pures et Appliqu{\'e}es},
  shortjournal={J. Math. Pures Appl. (9)},
  volume={81},
  number={9},
  pages={847--875},
  year={2002},
  publisher={Elsevier}
}

@article{Case15,
  title={A {Y}amabe-type problem on smooth metric measure spaces},
  author={Case, J. S.},
  journal={Journal of Differential Geometry},
  shortjournal={J. Differ. Geom.},
  volume={101},
  number={3},
  pages={467--505},
  year={2015},
  publisher={Lehigh University}
}

@article{GLZ26,
  title={Global Minimality of the {H}opf Map in the {F}addeev--{S}kyrme Model with Large Coupling Constant},
  author={Guerra, A. and Lamy, X. and Zemas, K.},
  journal={Communications in Mathematical Physics},
  shortjournal={Comm. Math. Phys.},
  volume={407},
  number={8},
  pages={171},
  year={2026},
  publisher={Springer}
}

@article{Ward99,
  title={Hopf solitons on $\mathbb{S}^3$ and $\mathbb{R}^3$},
  author={Ward, R. S.},
  journal={Nonlinearity},
  volume={12},
  number={2},
  pages={241--246},
  year={1999}
}

@article{SpeightSvensson07,
  title={On the strong coupling limit of the {F}addeev--{H}opf model},
  author={Speight, J. M. and Svensson, M.},
  journal={Communications in mathematical physics},
  shortjournal={Comm. Math. Phys.},
  volume={272},
  number={3},
  pages={751--773},
  year={2007},
  publisher={Springer}
}

@article{SpeightSvensson11,
  title={Some global minimizers of a symplectic {D}irichlet energy},
  author={Speight, J. M. and Svensson, M.},
  journal={Quarterly journal of mathematics},
  shortjournal={Q. J. Math.},
  volume={62},
  number={3},
  pages={737--745},
  year={2011},
  publisher={OUP}
}

@article{Riviere23,
  title={Harmonic maps from $\mathbb{S}^3$ into $\mathbb{S}^2$ with low Morse index},
  author={Rivi{\`e}re, T.},
  journal={J. Differ. Geom.},
  volume={125},
  number={1},
  pages={173--185},
  year={2023},
  publisher={Lehigh University}
}

@article{A03,
  title={A variational problem in conformal spin geometry},
  author={Ammann, B.},
  journal={Habilitationsschrift, Universit{\"a}t Hamburg},
  year={2003}
}

@article{Harland13MassivePions,
  title        = {Topological energy bounds for the {S}kyrme and {F}addeev models with massive pions},
  author       = {Harland, Derek},
  journal      = {Physics Letters B},
  shortjournal = {Phys. Lett. B},
  volume       = {728},
  pages        = {518--523},
  year         = {2014},
 % doi          = {10.1016/j.physletb.2013.11.062},
 % eprint       = {1311.2403},
 % archivePrefix= {arXiv},
 % primaryClass = {hep-th}
}

@article{LinYang07,
  title={Energy splitting, substantial inequality, and minimization for the {F}addeev and {S}kyrme models},
  author={Lin, F. and Yang, Y.},
  journal={Comm. Math. Phys.},
  volume={269},
  number={1},
  pages={137--152},
  year={2007},
  publisher={Springer}
}

@article{Gillard10,
  title={Hopf solitons in the {AFZ} model},
  author={Gillard, M.},
  journal={Nonlinearity},
  volume={24},
  number={10},
  pages={2729--2743},
  year={2011}
}

@article{Nicole78,
  title={Solitons with non-vanishing {H}opf index},
  author={Nicole, D. A.},
  journal={Journal of Physics G: Nuclear Physics},
  volume={4},
  number={9},
  pages={1363--1369},
  year={1978}
}

@article{AFZ99,
  title={Exact static soliton solutions of $(3+1)$-dimensional integrable theory with nonzero {H}opf numbers},
  author={Aratyn, H. and Ferreira, L. A. and Zimerman, A. H.},
  journal={Physical review letters},
  volume={83},
  number={9},
  pages={1723},
  year={1999},
  publisher={APS}
}

@article{Sutcliffe07,
  title={Knots in the {S}kyrme--{F}addeev model},
  author={Sutcliffe, P.},
  journal={Proceedings of the Royal Society A: Mathematical, Physical and Engineering Sciences},
  volume={463},
  number={2087},
  pages={3001--3020},
  year={2007},
  publisher={The Royal Society London}
}

@article{Faddeev,
  title={Some comments on the many-dimensional solitons},
  author={Faddeev, L. D.},
  journal={Letters in Mathematical Physics},
  volume={1},
  number={4},
  pages={289--293},
  year={1976},
  publisher={Springer}
}

@article{KunduRybakov82,
  title={Closed-vortex-type solitons with {H}opf index},
  author={Kundu, A. and Rybakov, Yu. P.},
  journal={Journal of Physics A: Mathematical and General},
  volume={15},
  number={1},
  pages={269--275},
  year={1982}
}

@article{IL93,
  title={Integral estimates for null {L}agrangians},
  author={Iwaniec, T. and Lutoborski, A.},
  journal={Archive for Rational Mechanics and Analysis},
  volume={125},
  number={1},
  pages={25--79},
  year={1993},
  publisher={Springer}
}

@article{Hamburger92,
author = {Hamburger, C.},
journal = {Journal für die reine und angewandte Mathematik},
pages = {7-64},
title = {Regularity of differential forms minimizing degenerate elliptic functionals.},
volume = {431},
year = {1992},
}

@article{Campanato63,
  title={Propriet{\`a} di h{\"o}lderianit{\`a} di alcune classi di funzioni},
  author={Campanato, S.},
  journal={Annali della Scuola Normale Superiore di Pisa-Scienze Fisiche e Matematiche},
  volume={17},
  number={1-2},
  pages={175--188},
  year={1963}
}

\end{document}